\documentclass[reqno, 12pt]{amsart}
\pdfoutput=1
\makeatletter
\let\origsection=\section \def\section{\@ifstar{\origsection*}{\mysection}}
\def\mysection{\@startsection{section}{1}\z@{.7\linespacing\@plus\linespacing}{.5\linespacing}{\normalfont\scshape\centering\S}}
\makeatother

\usepackage{amsmath,amssymb,amsthm}
\usepackage{mathrsfs}
\usepackage{mathabx}\changenotsign
\usepackage{dsfont}

\usepackage{xcolor}
\usepackage[backref]{hyperref}
\hypersetup{
    colorlinks,
    linkcolor={red!60!black},
    citecolor={green!60!black},
    urlcolor={blue!60!black}
}

\usepackage[open,openlevel=2,atend]{bookmark}

\usepackage[abbrev,msc-links,backrefs]{amsrefs}
\usepackage{doi}

\renewcommand{\PrintDOI}[1]{\doi{#1}}

\usepackage[T1]{fontenc}
\usepackage{lmodern}
\usepackage[babel]{microtype}
\usepackage[english]{babel}

\usepackage{geometry}
\numberwithin{equation}{section}
\numberwithin{figure}{section}

\usepackage{enumitem}
\def\rmlabel{\upshape({\itshape \roman*\,})}

\let\polishlcross=\l
\def\l{\ifmmode\ell\else\polishlcross\fi}

\let\emptyset=\varnothing
\let\setminus=\smallsetminus
\let\sm=\smallsetminus

\makeatletter
\def\moverlay{\mathpalette\mov@rlay}
\def\mov@rlay#1#2{\leavevmode\vtop{   \baselineskip\z@skip \lineskiplimit-\maxdimen
   \ialign{\hfil$\m@th#1##$\hfil\cr#2\crcr}}}
\newcommand{\charfusion}[3][\mathord]{
    #1{\ifx#1\mathop\vphantom{#2}\fi
        \mathpalette\mov@rlay{#2\cr#3}
      }
    \ifx#1\mathop\expandafter\displaylimits\fi}
\makeatother

\newcommand{\dcup}{\charfusion[\mathbin]{\cup}{\cdot}}

\DeclareFontFamily{U}  {MnSymbolC}{}
\DeclareSymbolFont{MnSyC}         {U}  {MnSymbolC}{m}{n}
\DeclareFontShape{U}{MnSymbolC}{m}{n}{
    <-6>  MnSymbolC5
   <6-7>  MnSymbolC6
   <7-8>  MnSymbolC7
   <8-9>  MnSymbolC8
   <9-10> MnSymbolC9
  <10-12> MnSymbolC10
  <12->   MnSymbolC12}{}
\DeclareMathSymbol{\powerset}{\mathord}{MnSyC}{180}

\usepackage{tikz}
\usetikzlibrary{calc,decorations.pathmorphing}
\pgfdeclarelayer{background}
\pgfdeclarelayer{foreground}
\pgfdeclarelayer{front}
\pgfsetlayers{background,main,foreground,front}

\usepackage{subcaption}

\newcommand{\redge}[8]{

		\ifx\relax#5\relax
		\def\qoffs{0pt}
	\else
		\def\qoffs{#5}
	\fi

				\def\rhedge{
		($#1+#4!\qoffs!-90:#2-#4$) -- 
		($#2+#1!\qoffs!-90:#3-#1$) -- 
		($#3+#2!\qoffs!-90:#4-#2$) -- 
		($#4+#3!\qoffs!-90:#1-#3$) -- cycle}

	\coordinate (12) at ($#1!\qoffs!90:#2$);
	\coordinate (14) at ($#1!\qoffs!-90:#4$);
	\coordinate (23) at ($#2!\qoffs!90:#3$);
	\coordinate (21) at ($#2!\qoffs!-90:#1$);
	\coordinate (34) at ($#3!\qoffs!90:#4$);
	\coordinate (32) at ($#3!\qoffs!-90:#2$);
	\coordinate (41) at ($#4!\qoffs!90:#1$);
	\coordinate (43) at ($#4!\qoffs!-90:#3$);
	
	\def\nrhedge{
		(14) let \p1=($(14)-#1$), \p2=($(12)-#1$) in 
			arc[start angle={atan2(\y1,\x1)}, delta angle={atan2(\y2,\x2)-atan2(\y1,\x1)-360*(atan2(\y2,\x2)-atan2(\y1,\x1)>0)}, x radius=\qoffs, y radius=\qoffs] --
		(21) let \p1=($(21)-#2$), \p2=($(23)-#2$) in 
			arc[start angle={atan2(\y1,\x1)}, delta angle={atan2(\y2,\x2)-atan2(\y1,\x1)-360*(atan2(\y2,\x2)-atan2(\y1,\x1)>0)}, x radius=\qoffs, y radius=\qoffs] --
		(32) let \p1=($(32)-#3$), \p2=($(34)-#3$) in 
			arc[start angle={atan2(\y1,\x1)}, delta angle={atan2(\y2,\x2)-atan2(\y1,\x1)-360*(atan2(\y2,\x2)-atan2(\y1,\x1)>0)}, x radius=\qoffs, y radius=\qoffs] --
		(43) let \p1=($(43)-#4$), \p2=($(41)-#4$) in 
			arc[start angle={atan2(\y1,\x1)}, delta angle={atan2(\y2,\x2)-atan2(\y1,\x1)-360*(atan2(\y2,\x2)-atan2(\y1,\x1)>0)}, x radius=\qoffs, y radius=\qoffs] --
		cycle}

		\ifx\relax#6\relax
		\def\rlwidth{1pt}
	\else
		\def\rlwidth{#6}
	\fi
	
		\ifx\relax#8\relax
		\fill \nrhedge;
	\else
		\fill[#8]\nrhedge;
	\fi

		\ifx\relax#7\relax
		\draw[line width=\rlwidth,rounded corners=\qoffs]\nrhedge;
	\else
		\draw[line width=\rlwidth,#7]\nrhedge;
	\fi
	}

\newcommand{\qedge}[7]{

	\ifx\relax#4\relax
		\def\qoffs{0pt}
	\else
		\def\qoffs{#4}
	\fi

	\def\qhedge{
		($#1+#3!\qoffs!-90:#2-#3$) --
		($#2+#1!\qoffs!-90:#3-#1$) --
		($#3+#2!\qoffs!-90:#1-#2$) -- cycle}

	\coordinate (12) at ($#1!\qoffs!90:#2$);
	\coordinate (13) at ($#1!\qoffs!-90:#3$);
	\coordinate (23) at ($#2!\qoffs!90:#3$);
	\coordinate (21) at ($#2!\qoffs!-90:#1$);
	\coordinate (31) at ($#3!\qoffs!90:#1$);
	\coordinate (32) at ($#3!\qoffs!-90:#2$);
	
	\def\nqhedge{
		(13) let \p1=($(13)-#1$), \p2=($(12)-#1$) in
			arc[start angle={atan2(\y1,\x1)}, delta angle={atan2(\y2,\x2)-atan2(\y1,\x1)-360*(atan2(\y2,\x2)-atan2(\y1,\x1)>0)}, x radius=\qoffs, y radius=\qoffs] --
		(21) let \p1=($(21)-#2$), \p2=($(23)-#2$) in
			arc[start angle={atan2(\y1,\x1)}, delta angle={atan2(\y2,\x2)-atan2(\y1,\x1)-360*(atan2(\y2,\x2)-atan2(\y1,\x1)>0)}, x radius=\qoffs, y radius=\qoffs] --
		(32) let \p1=($(32)-#3$), \p2=($(31)-#3$) in
			arc[start angle={atan2(\y1,\x1)}, delta angle={atan2(\y2,\x2)-atan2(\y1,\x1)-360*(atan2(\y2,\x2)-atan2(\y1,\x1)>0)}, x radius=\qoffs, y radius=\qoffs] --
		cycle}

		\ifx\relax#5\relax
		\def\qlwidth{1pt}
	\else
		\def\qlwidth{#5}
	\fi
	
		\ifx\relax#7\relax
		\fill \nqhedge;
	\else
		\fill[#7]\nqhedge;
	\fi

		\ifx\relax#6\relax
		\draw[line width=\qlwidth,rounded corners=\qoffs]\nqhedge;
	\else
		\draw[line width=\qlwidth,#6]\nqhedge;
	\fi
}

\let\epsilon=\varepsilon

\let\rho=\varrho
\let\theta=\vartheta

\let\phi=\varphi

\def\EE{{\mathds E}}
\def\NN{{\mathds N}}

\def\PP{{\mathds P}}

\newcommand{\cK}{\mathcal{K}}

\newcommand{\cR}{\mathcal{R}}

\newcommand{\cT}{\mathcal{T}}

\newcommand{\ccC}{\mathscr{C}}

\newcommand{\ccW}{\mathscr{W}}

\newtheoremstyle{note}  {4pt}  {4pt}  {\sl}  {}  {\bfseries}  {.}  {.5em}          {}
\newtheoremstyle{introthms}  {3pt}  {3pt}  {\itshape}  {}  {\bfseries}  {.}  {.5em}          {\thmnote{#3}}
\newtheoremstyle{remark}  {2pt}  {2pt}  {\rm}  {}  {\bfseries}  {.}  {.3em}          {}

\theoremstyle{plain}
\newtheorem{thm}{Theorem}[section]
\newtheorem{theorem}{Theorem}[section]
\newtheorem{lemma}[theorem]{Lemma}

\newtheorem{prop}[theorem]{Proposition}

\newtheorem{claim}[theorem]{Claim}

\theoremstyle{note}
\newtheorem{dfn}[theorem]{Definition}

\theoremstyle{remark}

\usepackage{accents}
\newcommand{\seq}[1]{\accentset{\rightharpoonup}{#1}}

\usepackage{lineno}
\newcommand*\patchAmsMathEnvironmentForLineno[1]{\expandafter\let\csname old#1\expandafter\endcsname\csname #1\endcsname
\expandafter\let\csname oldend#1\expandafter\endcsname\csname end#1\endcsname
\renewenvironment{#1}{\linenomath\csname old#1\endcsname}{\csname oldend#1\endcsname\endlinenomath}}\newcommand*\patchBothAmsMathEnvironmentsForLineno[1]{\patchAmsMathEnvironmentForLineno{#1}\patchAmsMathEnvironmentForLineno{#1*}}\AtBeginDocument{\patchBothAmsMathEnvironmentsForLineno{equation}\patchBothAmsMathEnvironmentsForLineno{align}\patchBothAmsMathEnvironmentsForLineno{flalign}\patchBothAmsMathEnvironmentsForLineno{alignat}\patchBothAmsMathEnvironmentsForLineno{gather}\patchBothAmsMathEnvironmentsForLineno{multline}}

\begin{document}
\title[Squares of Hamiltonian cycles in 3-uniform hypergraphs]{Squares of Hamiltonian 
cycles in 3-uniform hypergraphs}

\author[Wiebke Bedenknecht, Christian Reiher]{Wiebke Bedenknecht, Christian Reiher}
\thanks{The second author was supported by the European Research Council 
(ERC grant PEPCo 724903).}

\address{Fachbereich Mathematik, Universit\"at Hamburg,
  Bundesstra\ss{}e~55, D-20146 Hamburg, Germany}
\email{Wiebke.Bedenknecht@uni-hamburg.de}
\email{Christian.Reiher@uni-hamburg.de}
\keywords{hypergraphs, Hamiltonian cycles, P\'osa's conjecture}
\subjclass[2010]{Primary: 05C65. Secondary: 05C45.
}
\begin{abstract}
We show that every $3$-uniform hypergraph $H=(V,E)$ with $|V(H)|=n$ and minimum pair 
degree at least $(4/5+o(1))n$ contains a squared Hamiltonian cycle. This may be regarded
as a first step towards a hypergraph version of the P\'osa-Seymour conjecture.
\end{abstract} 

\maketitle

\section{Introduction} \label{sec:intro}
\subsection{Graphs} 
G. A. Dirac \cite{dirac} proved in 1952 that every graph $G=(V,E)$ with $|V|\geq 3$ and 
minimum vertex degree $\delta (G)\geq |V|/2$ contains a Hamiltonian cycle. 
Since on any set~$V$ of at least three vertices there are graphs $G$ with minimum degree 
$\delta (G)=\lceil |V|/2\rceil -1$, which do not contain a Hamiltonian cycle, 
this is an optimal result. 
Moreover, in 1962 P\'osa conjectured that every graph $G=(V,E)$ with $|V|\geq 5$ 
and minimum degree $\delta(G)\geq 2|V|/3$ contains the square of a Hamiltonian cycle.  
This conjecture was generalised further by Seymour~\cite{PDS} to the so-called {\it 
P\'osa-Seymour conjecture}, asking for the $k$-th power of a Hamiltonian cycle in graphs $G$ 
with $\delta(G)\geq \frac {k}{k+1}|V|$. Let us recall at this point that a graph is said to contain
the {\it $k$-th power of a Hamiltonian cycle} if its vertices can be arranged on a circle in such a way 
that any two vertices whose distance is at most $k$ are connected by an edge. 

A proof of this generalised conjecture for large graphs was obtained by Koml\'os, S\'ark\"ozy, 
and Szemer\'edi~\cite{KoSaSz98}. Their proof is based on the regularity method for graphs and 
uses the so-called {\it blow-up lemma} \cite{KoSaSz97} that was developed by the same authors shortly
before.  We will study an analogous P\'osa-type problem for $3$-uniform hypergraphs, i.e., what 
minimum pair-degree condition guarantees the existence of a squared Hamiltonian cycle?

\subsection{Hamiltonian cycles in hypergraphs}
A {\it $3$-uniform hypergraph} $H=(V,E)$ consists of a finite set $V=V(H)$ of vertices and a 
family $E=E(H)$ of $3$-element subsets of $V$, which are called {\it (hyper)edges}. 
Throughout this article if we talk about hypergraphs we will always mean $3$-uniform hypergraphs. 
We will write $xy$ and $xyz$ instead of $\{x,y\}$ and $\{x,y,z\}$ for edges and hyperedges. 
Similarly, we shall say that $wxyz$ is a {\it tetrahedron} or a~$K^{(3)}_4$ in a hypergraph $H$ 
if the triples $wxy$, $wxz$, $wyz$, and $xyz$ are edges of $H$.

There are at least two concepts of minimum degree and several notions of cycles like tight, 
loose and Berge cycles~\cite{berge} (see also~\cite{Bermond}). Here we will only introduce 
some of these notions.
 
If $H=(V,E)$ is a hypergraph and $v\in V$ is a vertex of $H$, then we denote by
\[
	d_H(v)=|\{ e\in E\colon v\in e\}|
\]
the {\it degree of $v$} and by
\[
	\delta_1(H)=\min \{d_H(v)\colon v\in V\}
\]
the {\it minimum vertex degree} of $H$ taken over all $v\in V$.

Similarly, for two vertices $u,v\in V$ we denote by
\[
	d_H(u,v)=|N_H(u,v)|=|\{ e\in E\colon u,v\in e\}|
\]
the {\it pair-degree of $u$ and $v$} and by
\[
	\delta_2(H)=\min \{d_H(u ,v)\colon uv\in V^{(2)}\}
\]
the {\it minimum pair-degree} of $H$ taken over all pairs of vertices of $H$.

We call a hypergraph $P$ a {\it tight path of length $\ell$}, if $|V(P)|=\ell+2$ 
and there exists an ordering of the vertices $V(P)=\{v_1,\ldots ,v_{\ell+2}\}$ 
such that a triple $e$ forms a hyperedge of~$P$ iff $e=\{v_i,v_{i+1},v_{i+2}\}$ 
for some $i\in [\ell]$. A {\it tight cycle $C$ of length $\ell \geq 4$} consists 
of a path $v_1\ldots v_{\ell}$ of length $\ell-2$ and the additional hyperedges 
$\{v_{\ell-1},v_{\ell },v_1\}$ and $\{ v_{\ell}, v_1,v_2\}$.  

A (tight) {\it Hamiltonian cycle} in a hypergraph $H$ is a tight cycle passing through 
all vertices of $H$. In other words, $H$ contains a Hamiltonian cycle if its vertices
can be arranged around a circle in such a way that any three consecutive vertices form 
an edge. The problem to determine optimal minimum degree conditions for hypergraphs, which 
enforce, as in Dirac's theorem, the existence of a Hamiltonian cycle, has received considerable 
attention. The first asymptotically optimal Dirac-type result for $3$-uniform hypergraphs 
was obtained by R{\"o}dl, Ruci{\'n}ski, and Szemer{\'e}di, who proved in \cite{rrs3} that 
every $n$-vertex hypergraph~$H$ with $\delta_2(H)\geq \bigl(\frac{1}{2}+o(1)\bigr)n$ contains 
a Hamiltonian cycle. In \cite{3} the same authors showed this for large~$n$ under the optimal assumption 
$\delta_2(H)\geq \lfloor n/2 \rfloor $. Moreover, it was proved in \cite{R3S2} that a 
minimum vertex degree condition of $\delta_1(H)\geq \bigl(\frac{5}{9}+o(1)\bigr)\frac{n^2}{2}$ 
guarantees the existence of a Hamiltonian cycle as well, where the constant $5/9$ is again 
best possible.  

\subsection{Squared Hamiltonian cycles in hypergraphs}
We call a hypergraph $P'$ a {\it squared path of length $\ell\geq 2$}, 
if $|V(P')|=\ell+2$ and there exists an ordering of the vertices 
$V(P')=\{v_1,\ldots ,v_{\ell+2}\}$ such that a triple~$e$ forms a hyperedge iff 
$e\subseteq \{v_i,v_{i+1},v_{i+2},v_{i+3}\}$ for some $i\in [\ell-1]$. 
In this case, we also say that $P'$ is a squared path {\it from} $(v_1, v_2, v_3)$ 
{\it to} $(v_{\ell}, v_{\ell+1}, v_{\ell+2})$, that $(v_1, v_2, v_3)$ and 
$(v_{\ell}, v_{\ell+1}, v_{\ell+2})$ are the {\it end-triples} of $P'$, and that 
the vertices $v_i$ with $4\le i\le \ell-1$ are the {\it internal vertices} of~$P'$.
{\it Squared walks} are defined like squared paths with the only difference that they 
are allowed to have repeated vertices.   

A {\it squared cycle~$C'$ of length~$\ell \geq 5$} consists of a squared path 
$v_1\ldots v_{\ell}$ of length~${\ell-2}$ and the additional hyperedges~$e$, which are $3$-subsets 
of at least one of the sets $\{v_{\ell-2},v_{\ell-1},v_{\ell },v_1\}$,  
$\{v_{\ell-1},v_{\ell },v_1, v_2\}$ or $\{ v_{\ell}, v_1,v_2,v_3\}$. 
As expected, a {\it squared Hamiltonian cycle} in a hypergraph~$H$ is a squared cycle passing 
through all vertices. 

Thus an $n$-vertex hypergraph $H$ contains a Hamiltonian squared cycle if its vertices 
can be arranged on a circle in such a way that every triple of vertices contained in an 
interval of length $4$ is an edge of $H$. Clearly this is a natural analogue of 
the concept of squared Hamiltonian cycles in graphs, where any pair contained in an interval 
of length $3$ is required to be an edge.

The main result of this article reads as follows.

\begin{thm}\label{thm:main}
For every $\alpha >0$ there exists an integer $n_0$ such that every $3$-uniform 
hypergraph~$H$ with $n\geq n_0$ vertices and with minimum pair-degree 
$\delta_2(H)\geq (\frac{4}{5}+\alpha )n$ contains a squared Hamiltonian cycle.
\end{thm}

We will denote by $K_4^{(3)}$ the complete $3$-uniform hypergraph on $4$ vertices. 
Note that any four consecutive vertices in a squared Hamiltonian cycle span a copy 
of~$K_4^{(3)}$. Therefore, if~$n$ is divisible by $4$, a squared Hamiltonian cycle contains 
a {\it $K_4^{(3)}$-tiling}, i.e., $\frac{n}{4}$ vertex disjoint copies of $K_4^{(3)}$. 
The problem to enforce $K_4^{(3)}$-tilings by an appropriate pair-degree condition was 
studied by Pikhurko \cite{Pi08}, who exhibited for every $n$ divisible by $4$ a hypergraph~$H$
on~$n$ vertices with $\delta_2(H)=\frac{3}{4}n-3$ not containing a $K_4^{(3)}$-tiling. 
Moreover, he proved that every $n$-vertex hypergraph $H$ with 
$\delta_2(H)\geq \big(\frac{3}{4}+o(1)\big)n$ contains vertex-disjoint copies of~$K_4^{(3)}$ 
covering all but at most $14$ vertices. We remark that based on Pikhurko's work~\cite{Pi08} 
the pair-degree problem for $K_4^{(3)}$-tilings was solved by Keevash and Mycroft 
in~\cite{KeMy15}. They showed that all $3$-uniform hypergraphs $H$ of sufficiently large 
order $n$ with $4\mid n$ and minimum pair-degree 
\[
	\delta_2(H)\geq 
	\begin{cases}
    	 3n/4-2 & \text{if } 8\mid n, \\
     	3n/4-1 & \text{otherwise } \\
   	\end{cases}
\]
contain a perfect $K_4^{(3)}$-tiling.

Notice that in view of Pikhurko's example the constant $\frac{4}{5}$ occurring in 
Theorem~\ref{thm:main} cannot be  replaced by anything below $\frac{3}{4}$ in case $4\mid n$. 
In order to extend this observation to all congruence classes modulo $4$ we take a closer look 
at the construction from \cite{Pi08}. Partition the vertex set 
$V=A_0\dcup A_1\dcup A_2\dcup A_3$ such that $\big | |A_i|-|A_j|\big |\leq 1$ for 
$0\leq i<j\leq 3$. Let $H$ be the hypergraph consisting of all the triples $e$ that satisfy 
one of the following properties (see Fig.~\ref{fig:Pikhurko}):
\begin{itemize}
\item $|A_0\cap e|=2$;
\item $e$ intersects each of $A_0, A_i, A_j$ for some $1\leq i<j\leq 3$;
\item $e\subseteq A_i$ for some $i\in [3]$;
\item $|e\cap A_i|=1$ and $|e\cap A_j|=2$ for some pair $ij\in [3]^{(2)}$.
\end{itemize}

\begin{figure}
\begin{tikzpicture}[scale=0.95]
	
	\coordinate (a) at (4,7);
	\coordinate (x) at (9,3);
	\coordinate (b) at (4,-1);
	\coordinate (z) at (-1,3);

	\coordinate (y1) at (4.25,6.35);
	\coordinate (y2) at (3.75,6.45);
	\coordinate (y3) at (3.75,6.25);
	
	\coordinate (w1) at (4,-0.5);
	\coordinate (w2) at (8,0.1);
	\coordinate (w3) at (0,0.1);
	
	\coordinate (k1) at (2,6);
	\coordinate (k2) at (0,5);
	\coordinate (k3) at (0,4.5);
	
	\coordinate (e1) at (6,6);
	\coordinate (e2) at (8,5);
	\coordinate (e3) at (8,4.5);
	
	\coordinate (f1) at (4,5.5);
	\coordinate (f2) at (3.75,0.5);
	\coordinate (f3) at (4.25,0.5);
	
	\coordinate (g) at (3.5,0.3);
   
	\begin{pgfonlayer}{front}
	
		\fill (y1) circle (2pt);
		\fill (y2) circle (2pt);
		\fill (y3) circle (2pt);
			
		\fill (w1) circle (2pt);
		\fill (w2) circle (2pt);
		\fill (w3) circle (2pt);
		
		\fill (k1) circle (2pt);
		\fill (k2) circle (2pt);
		\fill (k3) circle (2pt);
		
		\fill (e1) circle (2pt);
		\fill (e2) circle (2pt);
		\fill (e3) circle (2pt);
		
		\fill (f1) circle (2pt);
		\fill (f2) circle (2pt);
		\fill (f3) circle (2pt);
	
		\node at ($(a)+(0:0)$) {$A_0$};
		\node at ($(x)+(0:0)$) {$A_1$};
		\node at ($(b)+(0:0)$) { $A_2$};
		\node at ($(z)+(0:0)$) { $A_3$};
										
	\end{pgfonlayer}
	
	\begin{pgfonlayer}{background}
														
		\draw[cyan, line width=2pt] (4,6) ellipse (100pt and 20pt);
		\fill[blue, opacity=0.05] (4,6) ellipse (100pt and 20pt);
		
		\draw[cyan, line width=2pt] (4,0) ellipse (100pt and 20pt);
		\fill[blue, opacity=0.05] (4,0) ellipse (100pt and 20pt);
		
	\draw[cyan, line width=2pt] (8,3) ellipse (20pt and 90pt);
		\fill[blue, opacity=0.05] (8,3) ellipse (20pt and 90pt);
		
	\draw[cyan, line width=2pt] (0,3) ellipse (20pt and 90pt);
		\fill[blue, opacity=0.05] (0,3) ellipse (20pt and 90pt);
	\end{pgfonlayer} 

\qedge{(y1)}{(y3)}{(y2)}{4.5pt}{1.5pt}{red!70!white}{red!70!white,opacity=0.2};
\qedge{(w1)}{(w3)}{(w2)}{4.5pt}{1.5pt}{red!70!white}{red!70!white,opacity=0.2};
\qedge{(k1)}{(k3)}{(k2)}{4.5pt}{1.5pt}{red!70!white}{red!70!white,opacity=0.2};
\qedge{(e1)}{(e2)}{(e3)}{4.5pt}{1.5pt}{red!70!white}{red!70!white,opacity=0.2};
\qedge{(f1)}{(f3)}{(f2)}{4.5pt}{1.5pt}{red!70!white}{red!70!white,opacity=0.2};
	\end{tikzpicture}
	\caption {Complement of the hypergraph $H$, where the existing kinds of edges are indicated 
	in red, e.g. all triples with $3$ vertices in $A_0$ form an edge in the complement of $H$.}
	\label{fig:Pikhurko}
	\end{figure}
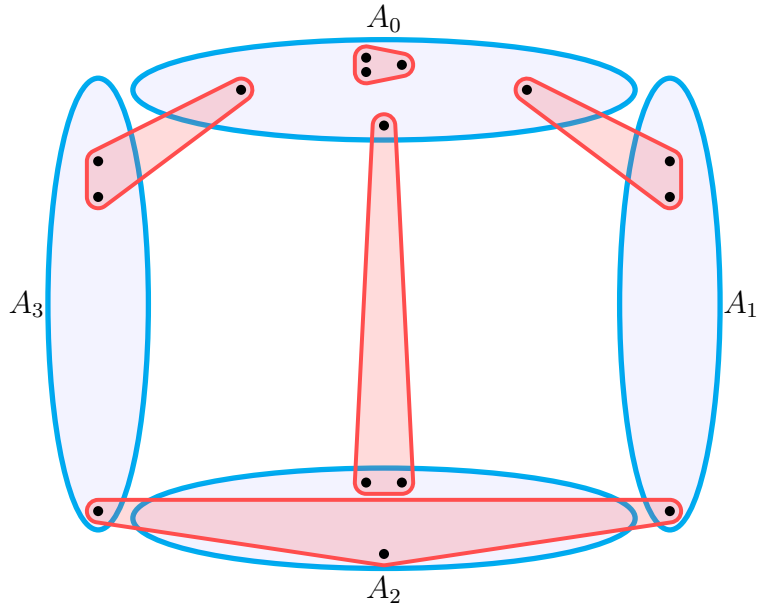	
Every $K_4^{(3)}$ intersecting $A_0$ has exactly $2$ vertices in $A_0$, since $A_0$ spans 
no edge and if a $K_4^{(3)}$ would intersect $A_0$ in only one vertex, then its remaining 
three vertices must come from $A_1$, $A_2$, $A_3$ (one from each set), but three such vertices 
do not form an edge in $H$. A squared Hamiltonian cycle $C\subseteq H$ needs to contain at least 
one $K_4^{(3)}$ that intersects $A_0$, but then each $K_4^{(3)}\subseteq C$ needs to 
intersect $A_0$ in two vertices. This implies $|A_0|\geq n/2$, which contradicts our 
assumption and shows that $H$ is indeed not containing a squared Hamiltonian cycle.

The proof of Theorem \ref{thm:main} is based on the {\it absorption method} developed 
by R{\"o}dl, Ruci{\'n}ski, and Szemer{\'e}di in \cite{3}. 
In Section \ref{sec:idea} we will discuss the general structure of the proof. 

\section{Building squared Hamiltonian Cycles in Hypergraphs}\label{sec:idea}

In this section we will show the outline of the proof of Theorem \ref{thm:main}. 
We start by presenting the dependencies of the auxiliary constants we use in the 
propositions required for this proof. We will sometimes 
exploit that the conclusion of Theorem~\ref{thm:main} is monotone in $\alpha$, i.e., 
that it becomes the stronger the smaller one takes $\alpha$ to be. Thus it will be 
permissible to assume that $\alpha$ is sufficiently small whenever convenient and we 
express this by writing that we may assume $1\gg \alpha$.  
More generally, $a\gg b$ will indicate that $b$ will be assumed to be sufficiently small 
depending on $a$ and all other constants appearing on the left of $b$. 

The connecting lemma stated below plays an important r\^{o}le in the proof of 
Theorem~\ref{thm:main}. Roughly speaking it asserts that any two disjoint triples 
of vertices can be connected by many ``short'' squared paths. 

\begin{prop}[Connecting Lemma]\label{p:cl}
	Given $\alpha>0$, 
	there are an integer $M$ and $\vartheta _*>0$, such that for all sufficiently large
	hypergraphs $H=(V,E)$ with $\delta_2(H)\geq (4/5+\alpha)|V|$ and all disjoint triples 
	$(a,b,c)$ and $(x,y,z)$ with $abc,xyz\in E$ there exists some $m<M$ for which there are at 
	least $\vartheta _* n^m$ squared paths from $abc$ to $xyz$ with $m$ internal vertices.
\end{prop}

We remark that the conclusion of the connecting lemma is monotone in $M$ and $\theta_*$.
Therefore we are free to assume that $M$ is sufficiently large (depending on $\alpha$)
and that $\theta_*$ is as small as we please (depending on $\alpha$ and $M$), i.e., that 
\[
	1\gg \alpha \gg 1/M \gg  \vartheta_* \,.
\]
The proof of the connecting lemma forms the content of Section~\ref{sec:con}. While we can connect 
any two squared paths by the connecting lemma using their start or endtriples, for our 
constructions it will be important that we do not interfere with any already constructed 
subpath. Therefore we put a small {\it reservoir} of vertices aside, such that if we do 
not connect too many times it is possible to use vertices of the reservoir set only. 
The following lemma, which we prove in Section \ref{sec:res}, shows the existence of 
such a set.

\begin{prop}[Reservoir Lemma]\label{p:res}
	Suppose that $1\gg \alpha \gg 1/M\gg \vartheta_*$ are such that the conclusion of the connecting 
	lemma holds and that $H=(V, E)$ is a sufficiently large hypergraph 
	with $|V|=n$ and $\delta_2(H)\geq (4/5+\alpha)n$. Then there exists a reservoir set 
	$\mathcal{R}\subseteq V$ of size $|\mathcal{R}|\leq \vartheta_*^2n$ such that for all 
	$\mathcal{R'}\subseteq \mathcal{R}$ with $|\mathcal{R'}|\leq \vartheta_{*}^4n $ and for 
	all disjoint triples $(a,b,c)$ and $(x,y,z)$ with $abc,xyz\in E$ there exists a connecting 
	squared path in $H$ with less than $M$ internal vertices all of which belong 
	to $\mathcal{R}\setminus \mathcal{R'}$.
\end{prop}

Moreover, we put aside an absorbing path $P_A$, which will absorb an arbitrary but not too 
large set $X$ of leftover vertices at the end of the proof, such that we get a squared 
Hamiltonian cycle.
 
\begin{prop}[Absorbing path] \label{p:abs}
	Let $1\gg \alpha \gg 1/M \gg \vartheta_*$ be such that the conclusion of the 
	connecting lemma holds, let $H=(V,E)$ be a sufficiently 
	large hypergraph with $|V|=n$ and $\delta_2(H)\geq (4/5+\alpha)n$,
	and let $\cR\subseteq V$ be a reservoir set as provided by Proposition~\ref{p:res}.
	There exists an (absorbing) squared path $P_A\subseteq H-\mathcal{R}$ such that 
	\begin{enumerate}
	\item  $|V(P_A)|\leq \vartheta_*n$,
	\item  and for every set $X\subseteq V\setminus V(P_A)$ with $|X|\leq 2\vartheta_*^2n$ there 
		is a squared path in $H$ whose set of vertices is $V(P_A)\cup X$ and whose end-triples
		are the same as those of~$P_A$.
	\end{enumerate}
\end{prop}

In Section~\ref{sec:abs} we prove Proposition \ref{p:abs} and in Section \ref{sec:cyc} we 
will show the following result.

\begin{prop}\label{p:alg}
	Given $\alpha, \mu >0 $ and $Q\in \NN$ there exists $n_0\in \NN$ such that 
	in every hypergraph~$H$ with $v(H)=n\geq n_0$ and $\delta_2(H)\geq (3/4+\alpha)n$ all but 
	at most $\mu n$ vertices of~$H$ can be covered by vertex-disjoint squared paths with $Q$ 	
	vertices.
\end{prop}

We conclude this section by proving that those four propositions do indeed imply Theorem~\ref{thm:main}.

\begin{proof}[Proof of Theorem~\ref{thm:main} based on Propositions~\ref{p:cl}--\ref{p:alg}]
	As already mentioned, we may assume that the given number $\alpha>0$ is sufficiently 
	small. Pick appropriate constants
	\[
		1\gg \alpha\gg 1/m\gg \theta_*\gg 1/n_0
	\]
	and let $H=(V, E)$ be a hypergraph with $|V|=n\ge n_0$ as well as $\delta_2(H)\ge (4/5+\alpha)n$.
	The reservoir lemma yields a reservoir set $\cR\subseteq V$ and then Proposition~\ref{p:abs}
	delivers an absorbing path $P_A\subseteq H-\cR$. The hypergraph $H'=H-(P_A\cup \mathcal{R})$
	satisfies 
	\[
		\delta_2(H')\ge \delta_2(H)-(|P_A|+\cR) \ge (4/5+\alpha-\theta_*-\theta_*^2)n 
		\ge  (4/5+\alpha/2)|V(H')|\,.
	\]
	
	So by Proposition~\ref{p:alg} applied with $\alpha/2$ here in place of $\alpha$ there, 
	with $\mu=\vartheta_*^2$ and with some $Q\geq M\vartheta _*^{-4}$ divisible by 4,
	there exists a family $\ccW$ of less than $n/Q$ disjoint squared paths in $H'$ the union 
	of which misses at most $\vartheta_*^2n$ vertices. 
	 
	Now we want to form a large squared cycle $\ccC$ by connecting the squared paths 
	in $\ccW\cup\{P_A\}$ through the reservoir. This is accomplished by $|\ccW|+1$
	successive applications of Proposition~\ref{p:res}. To see that this is possible 
	we note that even when the last connection is to be made, at most 
	$M|\ccW|\le Mn/Q\leq \vartheta_*^4 n$ vertices from the reservoir have already been used. 
	
	The vertices which are not in $\ccC$ are either unused vertices from the reservoir or they 
	were in $H'$ but not on any squared path in $\ccW$. Hence the set $X=V\setminus V(\ccC)$ satisfies 
	$|X|\le |\cR|+\vartheta_*^2n\le 2\vartheta_*^2n$. By Proposition~\ref{p:abs} it follows that 
	there exists a squared path $P^*$ in $H$ having the same end-triples as $P_A$ and whose set 
	of vertices is $V(P_A)\cup X$. Replacing $P_A$ by $P^*$ in $\ccC$ we obtain the desired 
	squared Hamiltonian cycle in $H$.  
\end{proof}

\section{Connecting Lemma}  \label{sec:con}

We will show some of our results with the constant $\frac{3}{4}$ and others for $\frac{4}{5}$.
Moreover we fix the auxiliary constants $\beta , \gamma , \vartheta_*$ and 
integers $K, \ell, M \in \NN$ obeying the hierarchy
\[
	1\gg \alpha \gg \beta , \gamma , 1/\ell \gg 1/K \gg 1/M \gg \vartheta_* \gg 1/n \,.
\]

\subsection{Connecting properties}
We prove that the graph properties stated in the following lemma imply a connecting property 
and use this lemma later to show that some auxiliary graphs $G_3$ and $G_v$ have this 
connecting property.

\begin{lemma}\label{lm:L16}
	Let $\gamma \leq 1/16$ and let $G=(V,E)$ with $|V|=n$ be a graph with 
	$\delta (G)\geq \sqrt {\gamma }n$ such that for every partition $X\dcup Y=V$ 
	of the vertex set with $|X|,|Y|\geq \sqrt{\gamma }n$ we have 
	$e_G(X,Y)\geq \gamma n^2$. 

	Then for every pair of distinct vertices $x,y\in V(G)$ there exists some 
	$s=s(x,y)\leq 4/\gamma$ for which there are at least $\Omega (n^{s-1})$ many
	$x$-$y$-walks of length $s$.
\end{lemma}

\begin{proof}
	For an arbitrary vertex $x\in V$ and  an integer $i\geq 1$ we define
		\[
		Z_x^i=\{ z \in V\colon \text{ there are at least } (\gamma^2/4)^s n^{s-1} 
			\text{ $x$-$z$-walks of length }s \text{ in } G \text{  for some }s\leq i\}\,.
	\]
		For $i\geq 2$ we have $Z_x^i \supseteq Z_x^{i-1}$ and therefore
	\[
		|Z_x^i|\geq |Z_x^1|= |N_G(x)|\geq \delta (G) \geq \sqrt{\gamma }n\, .
	\]

	Now we show that for every integer $i$ with $1\leq i\leq 2/\gamma $ at least one of 
	the following holds:
		\begin{equation}\label{eq:000}
		|V\setminus Z_x^i|<\sqrt{\gamma }n 
		\text{ \qquad or \qquad } 
		|Z_x^{i+1}\setminus Z_x^i|\geq \dfrac{\gamma n}{2}\,.
	\end{equation}
 		If $|V\setminus Z_x^i|\geq \sqrt{\gamma }n$, then the assumption yields that 
		\[
		e_G(Z_x^i,V\setminus Z_x^i)\geq \gamma n^2\,.
	\]
		This implies that at least $\gamma n/2$ vertices in $V\setminus Z_x^i$ have at least 
	$\gamma n/2$ neighbours in $Z_x^i$. For such a vertex $u\in V\setminus Z_x^i$ at least 
	a proportion of $1/i\geq \gamma/2$ of its neighbours in $Z_x^i$ is connected to~$x$ by 
	walks of the same length, which implies $u\in Z_x^{i+1}$. As this argument applies to 
	$\gamma n/2$ vertices outside $Z_x^i$ we thus obtain 
	$|Z_x^{i+1}\setminus Z_x^i|\geq \gamma n/2$, which concludes the proof of~\eqref{eq:000}.

 	It is not possible that the right outcome of \eqref{eq:000} holds for each positive 
	$i\leq 2/ \gamma$. Therefore we have 
	$|V\setminus Z_x^j|<\sqrt{\gamma }n$ for $j=\lfloor 2/\gamma \rfloor$.
	So for $x,y\in V$ at least $n-2\sqrt{\gamma }n\ge n/2$ vertices~$z$ are contained in the
	intersection $Z_x^j\cap Z_y^j$. For each $z \in Z_x^j\cap Z_y^j$ we get constants 
	$s_1,s_2\leq j\leq 2/\gamma $ such that there are at least $(\gamma^2/4)^{s_1}n^{s_1-1}$ 
	$x$-$z$-walks of length $s_1$ and there are at least $(\gamma^2/4)^{s_2}n^{s_2-1}$ 
	$z$-$y$-walks of length $s_2$. Therefore, for $s_z=s_1+s_2\geq 2$ there are at least 
	$(\gamma^2/4)^{s_z}n^{s_z-2}$ $x$-$y$-walks of length $s_z$ passing through $z$. 
	
	There are at least $n/2$ vertices this argument applies to and by the box principle
	at least $\frac{n}{2}/ \frac{4}{\gamma^2}$ of them give rise to the same 
	pair~$(s_1, s_2)$ and, consequently, the same value of $s_z$. Moreover, the walks 
	obtained for those vertices are distinct and hence for some~$s(x,y)\in [2, 4/\gamma]$ 
	there are at least
	\[
		(\gamma^2n/8) \cdot (\gamma^2/4)^{s(x,y)}n^{s(x,y)-2}
		\geq \tfrac 12(\gamma ^2/4)^{4/\gamma+1}n^{s(x,y)-1}
	\]
		$x$-$y$-walks of length $s(x,y)$.
\end{proof}

\subsection{The auxiliary graph \texorpdfstring{$G_3$}{}}
The first auxiliary graph we will study is the following.
\begin{dfn}\label{dfn:G3}
	For a 3-uniform hypergraph $H=(V,E)$ we define the auxiliary graph~$G_3$ 
	(see Fig.~\ref{fig:G3}) as the graph 
	with vertex set $V(G_3)=V$ and 
		\[
		xy\in E(G_3) \Longleftrightarrow x\neq y \text{ and } 
		\# \{(a,b,c)\in V^3\colon abcx \text{ and } abcy \text{ are } K_4^{(3)}  \}
		\geq \beta n^3 \, .
	\]
	\end{dfn}

Given a vertex $x$ of a hypergraph $H$ we denote its {\it link graph} by $L_x$. This is the 
graph with $V(L_x)=V(H)$ in which a pair $ab$ forms an edge if and only if $xab\in E(H)$.

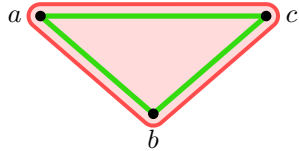
\begin{figure}[ht]
\begin{tikzpicture}[scale=1]
	
	\coordinate (x3) at (2.5,0.8);
	\coordinate (x4) at (1.4,3.4);
	
	\coordinate (z) at (4,-0.5);

	\coordinate (y3) at (5.5,0.8);
	\coordinate (y4) at (6.6,3.4);
   
	\begin{pgfonlayer}{front}
	
		\fill (z) circle (2pt);
		
		\foreach \i in {3,4}{
			\fill  (x\i) circle (2pt);
			\fill  (y\i) circle (2pt);
		}

		\node at ($(x3)+(180:0.34)$) {\footnotesize $a$};
		\node at ($(x4)+(180:0.34)$) {\footnotesize $x$};
		\node at ($(z)+(270:0.34)$) {\footnotesize $b$};
		\node at ($(y3)+(0:0.34)$) {\footnotesize $c$};
		\node at ($(y4)+(0:0.34)$) {\footnotesize $y$};
						
	\end{pgfonlayer}
	
	\begin{pgfonlayer}{background}
		\draw[green, line width=2pt] (x3) -- (z);
		\draw[green, line width=2pt] (z) -- (y3);
		\draw[green, line width=2pt] (y3) -- (x3);
		\draw[orange, line width=2pt, dotted] (x4) -- (y4);
	\end{pgfonlayer}

\qedge{(x3)}{(y3)}{(z)}{4.5pt}{1.5pt}{red!70!white}{red!70!white,opacity=0.2};

	\end{tikzpicture}
	\caption {We have an edge {\textcolor{orange}{$xy\in E(G_3)$} iff there are ``many'' 
	edges \textcolor{red}{$abc\in E(H)$} for which 
	\textcolor{green}{$ab,ac,bc\in E(L_x)\cap E(L_y)$}.}}
	\label{fig:G3}
\end{figure}

The main result of this subsection is the following proposition.

\begin{prop}\label{prop:033}
	Given $\alpha>0$ there exist $n_0, \ell \in \NN$ such that in every hypergraph~$H$
	with $v(H)=n\geq n_0$ and $\delta_2(H)\geq (3/4+\alpha )n$ for every pair of distinct 
	vertices $x,y\in V(G)$ there exists some $t=t(x,y)\leq \ell$ for which there are 
	at least $\Omega(n^{t-1})$  $x$-$y$-walks of length $t$ in $G_3$.
\end{prop} 

The next lemma gives us a lower bound on the minimum degree of $G_3$.

\begin{lemma}\label{lm:1}
	If $\alpha\gg n^{-1}$ and $H$ is a hypergraph on $n$ vertices with 
	$\delta_2(H)\geq (3/4+\alpha)n$, then $\delta (G_3)\geq (1/4+\alpha )n$.
\end{lemma}
 
\begin{proof}
	Let $x\in V$ and $\beta <\alpha/8$. We count the ordered quadruples $(a,b,c,y)\in V^4$, 
	such that $\{a,b,c,y\}$ and $\{x,a,b,c\}$ induce distinct tetrahedra in $H$. 
	That is, we estimate the size of the set 
		\[
		A_x=\{(a,b,c,y)\in V^4\colon x\neq y \text{ and } xabc ~\text{and}~ abcy 
		~\text{are} ~K_4^{(3)}\}\,.
	\]
		Due to our assumption about $\delta_2(H)$ the number $A$ of triples $(a,b,c)\in V^3$, 
	which form a~$K_4^{(3)}$ with $x$, can be estimated by
		\begin{align}\label{eq:Agro}
		A&=\# \{ (a,b,c)\in V^3\colon abcx ~\text{is a}~ K_4^{(3)} \}\notag \\
	& \geq (n-1)\Big (\dfrac{3n}{4}+\alpha n\Big)\Big( \dfrac{n}{4}+3\alpha n\Big)\notag \\
		&\geq \dfrac{n^3}{8}
	\end{align}
		for $n$ sufficiently large. Using the minimum pair-degree condition again we obtain
		\begin{equation}\label{eq:A11}
		|A_x|\geq A \Big(\dfrac{n}{4}+3\alpha n-1\Big)
		\geq \Big(\dfrac{1}{4}+2\alpha \Big) An\,.
	\end{equation}
		On the other hand, the assumption $d_{G_3}(x)\leq n/4+\alpha n$ would imply that
		\begin{align*}
 		|A_x|=\sum\limits_{y \in V \setminus \{ x\}} 
			\# \{(a,b,c)\in V^3\colon abcy ~\text{and}~ abcx ~\text{are}~ K_4^{(3)}\} 
		\leq n \cdot \beta n^3+(n/4+\alpha n )A \,.
	\end{align*}
		Together with \eqref{eq:A11} this yields that
		\[
		\Big(\dfrac{1}{4}+2\alpha \Big)An\leq \beta n^4+\Big (\dfrac{1}{4}+\alpha \Big )A n\,,
	\]
		i.e., $\beta n^3\geq \alpha A\overset{\eqref{eq:Agro}}{\geq} \alpha n^3/8$. 
	Since $\beta<\alpha/8$ this is a contradiction and shows that the minimum degree
 	of $G_3$ is at least $(1/4+\alpha )n$.
\end{proof}

\begin{lemma}\label{lm:2}
	If $\alpha\gg \beta, \gamma$ and $H$ is a hypergraph on $n$ vertices 
	with minimum pair-degree $\delta_2(H)\geq (3/4+\alpha)n$,
	then for every partition $X \dcup Y=V$ of the vertex set 
	with $|X|,|Y| \geq (1/4+\alpha/2)n$ we have $e_{G_3}(X,Y)\geq \gamma n^2$\,.
\end{lemma}

\begin{proof}
	W.l.o.g.  we can assume that $|X|\leq |Y|$. Since $|X|\geq (1/4+\alpha /2)n$, 
	we know that $|Y|\leq (3/4-\alpha /2)n$. Counting the ordered triples with two 
	vertices in $X$ and one in $Y$ which induce an edge in $H$, we get
	\begin{align*}
		& \#  \{(x,x',y)\in X^2\times Y\colon xx'y\in E(H)\}\\
		&=  \sum_{(x,y)\in X\times Y} |N(x,y)\cap X|\\
		&\geq |X||Y|\cdot (\delta_2(H)-|Y|)\\
		&\geq \dfrac{3}{16} n^2\cdot \dfrac{3\alpha n}{2} = \dfrac{9 \alpha}{32} n^3 \,.
	\end{align*}
	The number of $K_4^{(3)}$ including such a triple $(x,x',y)$ can thus be estimated by
	\begin{align*}
		\big|  \{(x,x', y,y')&\in X^2\times Y^2\colon xx'yy' \text{ is a } K_4^{(3)}\} \big|
		+\big|\{(x,x', x'',y)\in X^3\times Y\colon xx'x''y \text{ is a } K_4^{(3)}\}\big|\\
		&\geq \dfrac{9 \alpha n^3}{32}\cdot \dfrac{n}{4}=\dfrac{9\alpha}{128} n^4 \,.
	\end{align*}
	Now we will distinguish two cases depending on whether the number of $K_4^{(3)}$ 
	with exactly two or exactly three vertices in $X$ is bigger than $\frac{9\alpha}{256}n^4$.

	\smallskip

	{\it Case 1.} $\# \{(x,x', y,y')\in X^2\times Y^2\colon xx'yy' \text{ is a } K_4^{(3)}\} 
		\geq \dfrac{9\alpha}{256}n^4$

	\smallskip

	Define $A\subseteq X^2\times Y^2\times V$ to be the set of all quintuples 
	$(x,x',y,y',z)$ satisfying
	\begin{enumerate} [label=\rmlabel ,series=A]
		\item \label{it:31} $xx'yy'$ is a $K_4^{(3)}$;
		\item \label{it:32} $zxx',zyy'\in E(H)$;
		\item \label{it:33} and at least three of $zxy,zx'y',zxy',zx'y$ are edges in $H$.
	\end{enumerate}

	We claim that the size of $A$ can be bounded from below by
		\begin{equation} \label{eq:33}
		|A|\geq \dfrac{9\alpha^2}{64}n^5\,.
	\end{equation}
		Since we are in Case 1, it suffices to prove that every tetrahedron 
	$(x,x',y,y')\in X^2\times Y^2$ extends to at least $4\alpha n$ members of $A$.

	Writing 
		\[
		f(z)=| \{xy,xy',x'y,x'y'\}\cap E(L_z)|+2|\{ xx',yy'\} \cap E(L_z)|
	\]
		for every $z\in V$ we get 
		\begin{align*}
		\sum\limits_{z\in V} f(z)&
			=d _H (x,y)+d _H (x,y')+d _H (x',y)+d _H (x',y')+2d _H (x,x')+2d _H (y,y')\\
		&\geq 8 \delta_2(H)  \geq (6+8\alpha )n\,.
	\end{align*} 
		As $f(z)\leq 8$ holds for each $z\in V$ it follows that there are at least 
	$4\alpha n$ vertices with $f(z)\geq 7$. For each of them we have $(x,x',y,y',z)\in A$. 
	Thereby \eqref{eq:33} is proved.

	To derive an upper bound on $|A|$, we break the symmetry in \ref{it:33}. 
	Denoting by $A'$ the set of quintuples $(x,x',y,y',z)\in X^2\times Y^2\times V$ 
	satisfying~\ref{it:31},~\ref{it:32}, and
 	\begin{enumerate}[resume*=A]
		\item $xy'z, x'yz, x'y'z\in E(H)$
	\end{enumerate} 
	we have 
		\begin{equation}\label{eq:34}
		|A|\leq 4|A'|\, .
	\end{equation}
		Moreover
		\begin{align*}
		|A'|&\leq \sum_{(x,y)\in X\times Y} 
			\#\{(x',y',z)\in X\times Y\times V\colon xx'y'z \text{ and } x'yy'z \text{ are } 
				K_4^{(3)}\} \\
		&\leq e_{G_3}(X,Y)\cdot |X||Y||V|+|X||Y|\cdot \beta n^3\\
		&\leq \dfrac{1}{4} e_{G_3}(X,Y) n^3+\dfrac{1}{4}\beta n^5\,.
	\end{align*}
		Therefore with \eqref{eq:33} and \eqref{eq:34} it follows that
		\[
		e_{G_3}(X,Y)\geq \Big(\dfrac{9\alpha^2}{64}-\beta \Big )n^2 \,.
	\]
	
	\smallskip

	{\it Case 2.} $\#\{(x,x', x'',y)\in X^3\times Y\colon xx'x''y \text{ is a } K_4^{(3)}\} 
		\geq \dfrac{9\alpha}{256}n^4$

	\smallskip

	Define $A\subseteq X^3 \times Y \times V$ to be the set of all quintuples 
	$(x,x',x'',y,z)$ satisfying
	\begin{enumerate} [label=\rmlabel ,series=A]
		\item \label{it:311} $xx'x''y$ is a $K_4^{(3)}$;
		\item \label{it:322} if $z\in Y$ at least one of the vertex sets 
			$\{x,x'',y\}, \{x,x',y\}, \{x',x'',y\}$ induces a triangle in $L_z$;
		\item \label{it:333} if $z\in X$ the vertex set $\{x,x',x''\}$ induces a 
			triangle in $L_z$\,.
	\end{enumerate}
 
	We claim that the size of $A$ can be bounded from below by
		\begin{equation}\label{eq:35}
		|A|\geq \dfrac{27\alpha^2}{256}n^5\,.
	\end{equation}
		Since we are in Case 2, it suffices to prove that every tetrahedron 
	$(x,x',x'',y)\in X^3\times Y$ extends to at least $3\alpha n$ members of $A$.

	Writing 
		\[
		f(z)=| \{xy,xx',xx'',x'x'',x'y,x''y\}\cap E(L_z)|
	\]
		for every $z\in V$ we get 
		\begin{align*}
		\sum\limits_{z\in V} f(z)&
			=d _H (x,y)+d _H (x,x')+d _H (x,x'')+d _H (x',x'')+d _H (x',y)+d _H (x'',y)\\
	&\geq 6 \delta_2(H) \geq (9/2+6\alpha )n\,.
	\end{align*}
		If $z\in Y$ is a vertex with $(x,x',x'',y,z)\notin A$ then $f(z)\leq 4$ and if $z\in X$ 
	is a vertex with $(x,x',x'',y,z)\notin A$ then $f(z)\leq 5$. 
	Hence we have 
		\begin{align*}
		&(9/2+6\alpha)n \leq 5|X|+4|Y|+\big| \{z\in X\colon (x,x',x'', y,z)\in A\}\big|
			+2 \big| \{z\in Y\colon (x,x',x'', y,z)\in A\}\big|\,. 
	\end{align*}
		Since $5|X|+4|Y|=4n+|X|\leq 9/2 n$, it follows that 
		\[
		3\alpha n \leq \big| \{z\in X\colon (x,x',x'',y,z)\in A\}\big|
			+\big| \{z\in Y\colon (x,x',x'',y,z)\in A\}\big|\,,
	\]
		as claimed.

	Like before in Case~1 we obtain the upper bound 
		\begin{align*}
		&|A| \leq \beta n^5+e_{G_3}(X,Y)n^3\, .
	\end{align*}
		Therefore with \eqref{eq:35} it follows that
		\[ 
		e_{G_3}(X,Y)\geq \Big(\dfrac{  27\alpha^2}{256}-\beta \Big)n^2\,. \qedhere 
	\]
	\end{proof}

\begin{proof}[Proof of Proposition \ref{prop:033}]
	Because of Lemma~\ref{lm:L16}, Lemma~\ref{lm:1}, and Lemma \ref{lm:2}
	it remains to check that for every partition 
	$V=X\dcup Y$ with ${\sqrt{\gamma}n\le |X|\le (1/4+\alpha/2)n}$ 
	we have $e_{G_3}(X,Y)\ge \gamma n^2$.
	This follows easily from 
		\[
		e_{G_3}(X,Y)= \sum\limits_{x\in X} d^{G_3}_Y(x) \geq \delta (G_3) \cdot |X|- |X|^2
	\]
		and Lemma~\ref{lm:1}.  
\end{proof}

\subsection{The auxiliary graphs \texorpdfstring{$G_v$}{}} 
The second kind of auxiliary graphs we will study is the following.

\begin{dfn}\label{dfn:Gv}
	For a 3-uniform hypergraph $H=(V,E)$ and a vertex $v\in V$ we define the auxiliary 
	graph $G_v$ as the graph with vertex set $V(G_v)=V\setminus \{v\}$ and 
		\[
		xy\in E(G_v) \Longleftrightarrow x\neq y \text{ and } 
		\# \{(a,b)\in V^2\colon xabv \text{ and } yabv \text{ are } K_4^{(3)}  \}
		\geq \beta n^2 \,.
	\]
		\end{dfn}
	
\begin{figure}[ht]
\begin{tikzpicture}[scale=1.2]
	
	\coordinate (a) at (2.5,0.8);
	\coordinate (x) at (1.4,3.4);
	\coordinate (z) at (4,5);
	\coordinate (b) at (5.5,0.8);
	\coordinate (y) at (6.6,3.4);
	\coordinate (g) at (3.5,0.3);
   
	\begin{pgfonlayer}{front}
		\draw[green, line width=2pt] (a) -- (x);
		\draw[green, line width=2pt] (a) -- (b);
		\draw[green, line width=2pt] (b) -- (x);
		\draw[green, line width=2pt] (b) -- (y);
		\draw[green, line width=2pt] (a) -- (y);	
	\draw[orange, line width=2pt, dotted] (x) -- (y);
	
		\fill (z) circle (2pt);
		\fill (a) circle (2pt);
		\fill (x) circle (2pt);
		\fill (y) circle (2pt);
		\fill (b) circle (2pt);

		\node at ($(a)+(180:0.34)$) {\footnotesize $a$};
		\node at ($(x)+(180:0.34)$) {\footnotesize $x$};
		\node at ($(z)+(90:0.34)$) {\footnotesize $v$};
		\node at ($(b)+(0:0.34)$) {\footnotesize $b$};
		\node at ($(y)+(0:0.34)$) {\footnotesize $y$};
		\node at ($(g)+(0:0.34)$) {$V\setminus\{v\}$};
	\end{pgfonlayer}
	
	\begin{pgfonlayer}{background}

		\draw[cyan, line width=2pt] (4,2.1) ellipse (128pt and 60pt);
		\fill[blue, opacity=0.05] (4,2.1) ellipse (128pt and 60pt);
	\end{pgfonlayer} 

\qedge{(b)}{(a)}{(x)}{4.5pt}{1.5pt}{red!70!white}{red!70!white,opacity=0.2};
\qedge{(b)}{(a)}{(y)}{4.5pt}{1.5pt}{red!70!white}{red!70!white,opacity=0.2};
	\end{tikzpicture}
	\caption {We have {\textcolor{orange}{$xy\in E(G_v)$} iff there are ``many'' pairs 
	$(a,b)\in V^2$ for which \textcolor{red}{ $abx,aby \in E(H)$} and 
	\textcolor{green}{$abx, aby$} span triangles in \textcolor{green}{$L_v$}.}}
	\label{fig:Gv}
	\end{figure}
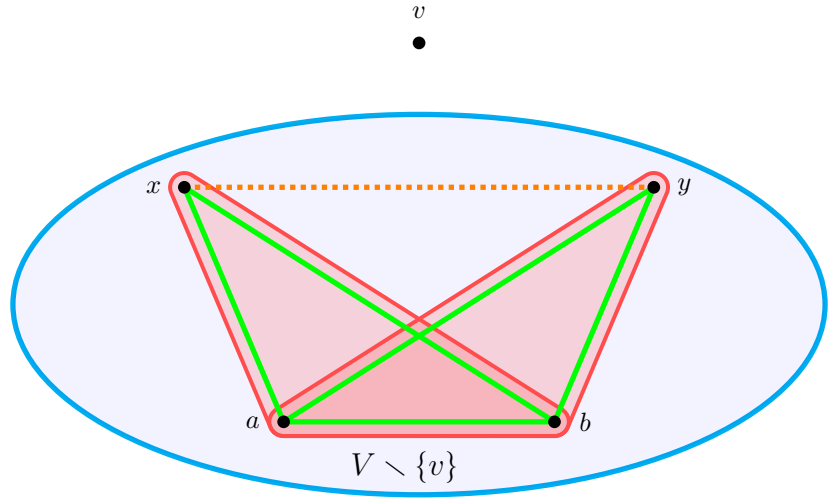	

The main result of this subsection is the following proposition.

\begin{prop}\label{prop:zwei}
	Given $\alpha>0$ there exist $n_0, \ell \in \NN$ such that in every 
	hypergraph $H$ with $v(H)=n\geq n_0$ and $\delta_2(H)\geq (3/4+\alpha )n$ for 
	every $v\in V(H)$ and for every pair of distinct vertices $x,y\in V(G_v)$ 
	there exists some $t=t(x,y)\leq \ell$ 
	for which there are at least $\Omega(n^{t-1})$  $x$-$y$-walks of length $t$ in $G_v$.
\end{prop}	
	
The next lemma gives us a lower bound on the minimum degree of $G_v$.

\begin{lemma}\label{lm:11}	
	If $\alpha\gg n^{-1}$ and $H$ is a hypergraph on $n$ vertices with 
	$\delta_2(H)\geq (3/4+\alpha)n$, then $\delta (G_v)\geq (1/4+\alpha )n$\,.
\end{lemma} 

\begin{proof}
	Let $x\in V\setminus \{v\}$. We count the triples $(a,b,y)\in V^3$, such that 
	$\{y,a,b,v\}$ and $\{x,a,b,v\}$ induce distinct tetrahedra in $H$. That is, we estimate 
	the size of the set 
		\[
		A_x=\{(a,b,y)\in V^3\colon x\neq y \neq v\text{ and } xabv ~\text{and}~ yabv 
		~\text{are} ~K_4^{(3)}\} \,.
	\]
		Due to our assumption about $\delta_2(H)$ the number $A$ of pairs $(a,b)\in V^2$, 
	which form a $K_4^{(3)}$ with $x$ and $v$, can be estimated by
		\begin{align}\label{eq:Agro2}
		A&=\# \{ (a,b)\in V^2\colon abxv ~\text{is a}~ K_4^{(3)} \}\notag \\
		& \geq \Big (\dfrac{3n}{4}+\alpha n\Big)\Big( \dfrac{n}{4}+3\alpha n\Big)\notag \\
				&\geq \dfrac{n^2}{8}\,.
	\end{align}
	Moreover we have
		\begin{equation}\label{eq:36}
		|A_x|\geq A \Big(\dfrac{n}{4}+3\alpha n-1\Big)
		\geq \Big (\frac{1}{4}+2\alpha \Big )An\,.
	\end{equation} 
		On the other hand, the assumption $d_{G_v}(x)\leq n/4+\alpha n$ would imply that
		\[
 		|A_x|=\sum\limits_{y\in V\setminus \{v,x\}} 
		\# \{(a,b)\in V^2\colon abvy ~\text{and}~ abvx ~\text{are}~ K_4^{(3)}\}
 	\leq n\cdot \beta n^2+(n/4+\alpha n)A\,.
	\]
		Together with \eqref{eq:36} this yields that
		\[
		\Big (\frac{1}{4}+2\alpha \Big )An\leq \beta n^3+\Big(\frac{1}{4}+\alpha  \Big )An\,,
	\]
		i.e., $\beta n^2\geq \alpha A \overset{\eqref{eq:Agro2}}{\geq} \alpha n^2/8$. 
	Since $\beta< \alpha/8$ this is a contradiction and shows that the minimum degree 
	of $G_v$ is at least $(1/4+\alpha )n$.
\end{proof}

\begin{lemma}\label{lm:21}
	If $\alpha\gg \beta, \gamma\gg  n^{-1}$ and $H$ is a hypergraph on $n$ vertices 
	with minimum pair-degree $\delta_2(H)\geq (3/4+\alpha)n$,
	then for every partition $X \dcup Y=V\setminus\{v\}$ of the vertex set 
	with $|X|,|Y| \geq (1/4+\alpha/2)n$ we have $e_{G_v}(X,Y)\geq \gamma n^2$\,.
\end{lemma}

\begin{proof}
	We begin by showing that the set 
		\[ 
		A_\star 
		=
		\{ (x,y,z)\in X\times Y\times (V\setminus\{v\})\colon vxyz \text{ is a } K_4^{(3)} 
			\text{ in } H\} \,,
	\]
		satisfies 
		\begin{equation}\label{eq:Astern}
		|A_\star |\geq \frac{n^3}{32}\, .
	\end{equation}
		For the proof of this fact we may assume that $|X|\leq |Y|$. 
	Thus $|X|\in [\frac{n}{4},\frac{n}{2}]$ and hence 
		\begin{align*}
		|A_\star | &\geq |X| \cdot (\delta_2(H)-|X|) \cdot (3\delta_2(H)-2n)\\
		&\geq |X| \cdot \Big (\dfrac{3}{4}n-|X|\Big )\cdot  \dfrac{n}{4}\\
		&\geq \dfrac{n^2}{8}\cdot \dfrac{n}{4}=\dfrac{n^3}{32}\, ,
	\end{align*}
		as desired.

	It follows that
	\[ 
		|A_\star \cap (X\times Y\times X)|+|A_\star \cap (X\times Y^2)|
		=|A_\star|\geq \dfrac{n^3}{32} 
	\]
	and w.l.o.g. we can assume that $|A_\star \cap (X\times Y\times X)|\geq n^3/64$.
	Now we study the set
	\[
		A_{\star \star}=
		\{(a,b,y,z)\in X^2\times Y\times (V\setminus\{v\})\colon abvy, abvz \text{ are } 
			K^{(3)}_4 
			\text{ and } yz\in E(L_v)\} \,.
	\]
	Given any triple $(a,y,b)\in A_\star\cap (X\times Y\times X)$ the quadruple 
	$abvy$ forms a tetrahedron, there are at least $3\delta_2(H)-2n$ vertices $z$ 
	for which $abvz$ forms a tetrahedron as well, and for at most $n-\delta_2(H)$ of those 
	the condition $yz\in E(L_v)$ fails. 
	Hence 
	\begin{align*}
		|A_{\star \star }| &\geq |A_\star \cap (X\times Y\times X) | 
			\cdot [(3\delta_2(H)-2n)-(n-\delta_2(H))]\\
	&\geq 4\alpha n \cdot |A_\star \cap (X\times Y\times X)|\geq \dfrac{\alpha}{16	}n^4\, .
	\end{align*}
		
	\smallskip
	
	{\it Case 1.} $|A_{\star \star}\cap (X^2\times Y \times X)|\geq \alpha n^4/32$.
	
	\smallskip
	
	Owing to
		\begin{align*}
		\dfrac{\alpha n^4}{32}&\leq |A_{\star \star}\cap (X^2\times Y \times X)|\\
		&\leq \sum \limits_{(z,y)\in X\times Y} \# \{(a,b)\in X^2\colon abzv \text{ and } abvy 
			\text{ are } K_4^{(3)}\}\\
		&\leq \beta n^2 |X||Y|+e_{G_v}(X,Y)\cdot n^2\\
		&\leq \beta n^2 \cdot n^2/4+e_{G_v} (X,Y)\cdot n^2
	\end{align*}
		we have
		\[
		e_{G_v}(X,Y)\geq \Big (\dfrac{\alpha}{32}-\dfrac{\beta }{4}\Big )n^2 \,,
	\]
		as desired.

	\smallskip

	{\it Case 2.} $|A_{\star \star}\cap (X^2\times Y^2)| \geq \alpha n^4/32$

	\smallskip

	Define $A\subseteq X^2\times Y^2\times (V\setminus \{v\})$ to be the set of all 
	quintuples $(x,x',y,y',z)$ satisfying
	\begin{enumerate}[label=\rmlabel]
		\item \label{it:eins} $xx'yy'$ is a $K_4$ in $L_v$
		\item at least one of $xx', yy'$ forms a $K_4^{(3)}$ with $v$ and $z$
		\item at least one of $xy, xy', x'y, x'y'$ forms a $K_4^{(3)}$ with $v$ and $z$.
		\end{enumerate} 

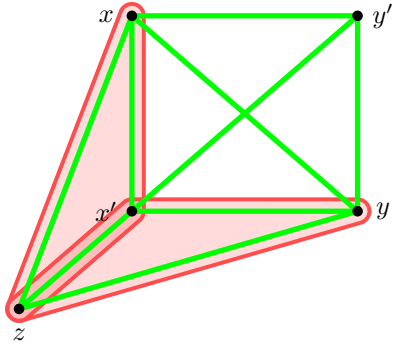
\begin{figure}
\begin{tikzpicture}[scale=1.0]
	
	\coordinate (a) at (2.5,0.8);
	\coordinate (x) at (2.5,3.4);
	
		\coordinate (z) at (1,-0.5);

	\coordinate (b) at (5.5,0.8);
	\coordinate (y) at (5.5,3.4);
	\coordinate (g) at (3.5,-1.5);

	\begin{pgfonlayer}{front}
		\draw[green, line width=2pt] (a) -- (x);
		\draw[green, line width=2pt] (a) -- (b);
		\draw[green, line width=2pt] (b) -- (x);
		\draw[green, line width=2pt] (b) -- (y);
		\draw[green, line width=2pt] (a) -- (y);
		\draw[green, line width=2pt] (a) -- (z);
		\draw[green, line width=2pt] (b) -- (z);
		\draw[green, line width=2pt] (x) -- (z);
		\draw[green, line width=2pt] (x) -- (y);
		\fill (a) circle (2pt);		
		\fill (b) circle (2pt);
		\fill (x) circle (2pt);
		\fill (z) circle (2pt);
		\fill (y) circle (2pt);
		
		\node at ($(a)+(180:0.34)$) {\footnotesize $x'$};
		\node at ($(x)+(180:0.34)$) {\footnotesize $x$};
		\node at ($(b)+(0:0.34)$) {\footnotesize $y$};
		\node at ($(y)+(0:0.34)$) {\footnotesize $y'$};
				\node at ($(z)+(270:0.34)$) {\footnotesize $z$};
						
	\end{pgfonlayer}
	\qedge{(b)}{(z)}{(a)}{4.5pt}{1.5pt}{red!70!white}{red!70!white,opacity=0.2};
\qedge{(x)}{(a)}{(z)}{4.5pt}{1.5pt}{red!70!white}{red!70!white,opacity=0.2};
	\end{tikzpicture}
	\caption {Example of a quintuple in $A$, where the \textcolor{green}{ link graph of $v$} 
	is indicated in green and \textcolor{red}{hyperedges of $H$} in red.}
	\label{fig:33}
	\end{figure}
	
	Notice that condition \ref{it:eins} holds for every 
	$(x,x',y,y')\in A_{\star \star}\cap (X^2\times Y^2)$.
	Let us now fix some such quadruple $(x,x',y,y')$.
 	Due to our assumption about $\delta _2 (H)$ we have 
		\begin{align*}
		d _H (x,y)+d _H (x,y')+d _H (x',y)+d _H (x',y')+2\bigl(d _H (x,x')+d _H (y,y')\bigr)\\
		+2\bigl(d_{L_v}(x)+d_{L_v}(x')+d_{L_v}(y)+d_{L_v}(y')\bigr)
		\geq 16 \delta_2(H)\geq (12+16\alpha)n\,.
	\end{align*}
		So writing 
		\[
		f(z)=| \{xy,xy',x'y,x'y'\}\cap E(L_z)|+2|\{ xx',yy'\} \cap E(L_z)|
			+2|\{vx,vx',vy,vy'\} \cap E(L_z)|
	\]
		for every $z\in V$ we get 
		\[
		 \sum\limits_{z\in V} f(z) \geq (12+16\alpha )n\,.
	\]
		If $z$ is a vertex with $(x,x',y,y',z)\notin A$, then $f(z)\leq 12$, and hence we have 
		\[
		\# \{z\in V\colon (x,x',y,y',z)\in A\}\geq 16\alpha n/4=4\alpha n \,.
	\]
		Applying this argument to every $(x,x',y,y')\in A_{\star \star}\cap (X^2\times Y^2)$      		we obtain, since we are in Case~2, that
	\begin{equation}\label{eq:A}
		|A|\geq \dfrac{\alpha}{32}n^4 \cdot 4\alpha n=\dfrac{\alpha^2}{8} n^5\,.
	\end{equation}
		
	Now let $A_x$ (resp. $A_y$) be the number of quintuples 
	$(x,x',y,y',z)\in X^2\times Y^2 \times (V\setminus \{v\})$ such that
	\begin{itemize}
		\item $xx'vz$ (resp. $yy'vz$) and $x'yvz$ are $K_4^{(3)}$. 
	\end{itemize}
	By symmetry we have 
		\[
		A_x+A_y\geq \dfrac{1}{4}|A|\overset{\eqref{eq:A}}{\geq} \dfrac{\alpha ^2}{32}n^5 \,.
	\]
		Consequently at least one of $A_x, A_y$ is at least $\frac{\alpha^2}{64}n^5$. 
	In either case one can prove that $e_{G_v}(X,Y)\geq \gamma n^2$ and below we display 
	the argument assuming $A_x\geq \frac{\alpha^2}{64}n^5$. In this case
 		\begin{align*}	
		A_x&\leq \sum \limits_{(x,y)\in X\times Y}  \# \{(x',y',z)\in V^3\colon  xx'zv 
		\text{ and } yx'zv \text{ are } K^{(3)}_4\} \\
		&\le n \sum \limits_{(x,y)\in X\times Y}  \# \{(x',z)\in V^2\colon  xx'zv 
		\text{ and } yx'zv \text{ are } K^{(3)}_4\} \\
		&\leq |X||Y|\beta n^3+e_{G_v}(X,Y)n^3
	\end{align*}
		yields
		\[
		e_{G_v}(X,Y)\geq \Big (\dfrac{\alpha^2}{64}-\dfrac{\beta}{4}\Big )n^2 \,, 
	\]
		as desired. The case $A_y\geq \frac{\alpha^2}{64}n^5$ is similar. 
\end{proof}

\begin{proof}[Proof of Proposition \ref{prop:zwei}]
	Because of Lemma \ref{lm:11}  and the fact that 
		\[
		e_{G_v}(X,Y)= \sum\limits_{x\in X} d^{G_v}_Y(x) \geq \delta (G_v) \cdot |X|- |X|^2\, ,
	\]
	 	Lemma~\ref{lm:21} is already true if $|X|,|Y|\geq \sqrt{\gamma} n$. Therefore the 
	assumptions of Lemma \ref{lm:L16} hold for the graph $G_v$, which implies 
	Proposition~\ref{prop:zwei}.  
\end{proof}

\subsection{Connecting Lemma}

For the rest of this section we will use the constant $\frac{4}{5}$, i.e., the minimum 
pair-degree hypothesis $\delta_2(H)\geq (4/5+\alpha )n$.

\begin{dfn}
	For a 3-uniform hypergraph $H=(V,E)$ and vertices $v,r,s\in V$ we 
	write 
		\[
		N_v(r,s)=N(r,s,v) = N(r,v)\cap N(s,v)\cap N(r,s) \,.
	\]
	\end{dfn}

Notice that our minimum pair-degree condition entails
\begin{equation}\label{eq:Nv}
	|N_v(r,s)|\ge 2n/5\ge n/4
\end{equation}
for all $v,r,s\in V$.

\begin{dfn}
	Given a hypergraph $H$ on $n$ vertices with minimum pair-degree 
	$\delta_2(H)\geq (4/5+\alpha)n$ and two distinct vertices $v, w\in V(H)$ we define 
	the auxiliary graph~$G_{vw}$ by $V(G_{vw})=N(v,w)$ and 
	\[
		uu'\in E(G_{vw}) \Longleftrightarrow uu'vw \text{ is a } K_4^{(3)}\,.
	\]
	\end{dfn}

Due to our assumption about the minimum pair-degree we know that the size $n'$
of the vertex set satisfies $n'=|V(G_{vw})|\geq (4/5+\alpha )n$.

\begin{lemma}\label{lm:B2}
	Let $v,w\in V$ and $b,x\in V(G_{vw})$. There are at least $\alpha n^2/2$ walks of 
	length $3$ from $b$ to $x$ in $G_{vw}$. 
\end{lemma}

\begin{proof}
	For a vertex $r\in V(G_{vw})$ we have
		\begin{align*}
		d_{G_{vw}}(r)&\geq  |V(G_{vw})|- 2(n-\delta_2(H))\\
		&\geq \dfrac{|V(G_{vw})|}{2}+\dfrac{\delta_2(H)}{2}- 2(n-\delta_2(H))\\
		&= \dfrac{|V(G_{vw})|}{2}+\dfrac{5\delta_2(H)}{2}- 2n
			\geq \dfrac{n'}{2}+\dfrac{5\alpha n}{2}\geq \Big(\dfrac{1}{2}+\alpha \Big ) n'\, . 
	\end{align*}
		
	Thus the minimum degree of $G_{vw}$ can be bounded from below by 
	$\delta(G_{vw})\geq (1/2+\alpha )n'$ and any two vertices of $G_{vw}$ have at least 
	$2\alpha n'$ common neighbours in $G_{vw}$. Due to this and the minimum vertex degree
	condition in $G_{vw}$ we can therefore find at least 
		\[
		\dfrac{n'}{2}\cdot 2\alpha n'=\alpha (n')^2\geq \dfrac{\alpha }{2}n^{2}
	\]
		walks of length $3$ from $b$ to $x$ in $G_{vw}$. 
 	This shows Lemma \ref{lm:B2}.
\end{proof}

\begin{lemma}\label{cl:1}
	If $vbc, vxy\in E$ and $|N_v(b,c)\cap N_v(x,y)|=m$, then there are at 
	least $\alpha^2 m^{2}n^{2}/4$ quadruples 
	$(w_0,b_1,c_1,w_1)$ such that $bcw_0b_1c_1w_1xy$ is
	\begin{itemize} 
	\item  a walk in $H$ and
	\item  a squared walk in $L_v$\,.
	\end{itemize} 
\end{lemma}
	
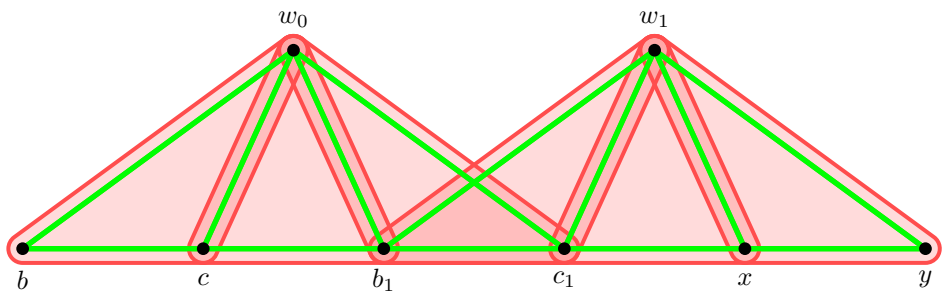
\begin{figure}[ht]
\begin{tikzpicture}[scale=1.2]
	
	\coordinate (b) at (-6,0.8);
	\coordinate (c) at (-4,0.8);
	\coordinate (b1) at (-2,0.8);
	\coordinate (c1) at (0,0.8);
	\coordinate (x) at (2,0.8);
	\coordinate (y) at (4,0.8);
	
	\coordinate (w0) at (-3,3);
	\coordinate (w1) at (1,3);

	\begin{pgfonlayer}{front}
	
\draw[green, line width=2pt] (b) -- (c);
		\draw[green, line width=2pt] (c) -- (b1);
		\draw[green, line width=2pt] (b1) -- (c1);
		\draw[green, line width=2pt] (c1) -- (x);
		\draw[green, line width=2pt] (b) -- (w0);
		\draw[green, line width=2pt] (c) -- (w0);
		\draw[green, line width=2pt] (b1) -- (w0);
		\draw[green, line width=2pt] (c1) -- (w0);
		\draw[green, line width=2pt] (b1) -- (w1);
		\draw[green, line width=2pt] (c1) -- (w1);
		\draw[green, line width=2pt] (x) -- (w1);
		\draw[green, line width=2pt] (y) -- (w1);
				\draw[green, line width=2pt] (x) -- (y);		
	
		\fill (b) circle (2pt);
		
		\fill (c) circle (2pt);
		
		\fill (b1) circle (2pt);
		
		\fill (c1) circle (2pt);
		
		\fill (x) circle (2pt);
		\fill (y) circle (2pt);
		\fill (w0) circle (2pt);
		\fill (w1) circle (2pt);

	\node at ($(w1)+(90:0.34)$) {\footnotesize $w_1$};
	\node at ($(w0)+(90:0.34)$) {\footnotesize $w_0$};
		\node at ($(b)+(270:0.34)$) {\footnotesize $b$};
		\node at ($(c)+(270:0.34)$) {\footnotesize $c$};
		\node at ($(b1)+(270:0.34)$) {\footnotesize $b_1$};
		\node at ($(c1)+(270:0.34)$) {\footnotesize $c_1$};
		\node at ($(x)+(270:0.34)$) {\footnotesize $x$};
				\node at ($(y)+(270:0.34)$) {\footnotesize $y$};

	\end{pgfonlayer}
	
	\begin{pgfonlayer}{background}
				\draw[green, line width=2pt] (b) -- (c);
		\draw[green, line width=2pt] (c) -- (b1);
		\draw[green, line width=2pt] (b1) -- (c1);
		\draw[green, line width=2pt] (c1) -- (x);
		\draw[green, line width=2pt] (b) -- (w0);
		\draw[green, line width=2pt] (c) -- (w0);
		\draw[green, line width=2pt] (b1) -- (w0);
		\draw[green, line width=2pt] (c1) -- (w0);
		\draw[green, line width=2pt] (b1) -- (w1);
		\draw[green, line width=2pt] (c1) -- (w1);
		\draw[green, line width=2pt] (x) -- (w1);
		\draw[green, line width=2pt] (y) -- (w1);
				\draw[green, line width=2pt] (x) -- (y);
					\end{pgfonlayer} 

\qedge{(c)}{(w0)}{(b1)}{4.5pt}{1.5pt}{red!70!white}{red!70!white,opacity=0.2};
\qedge{(c1)}{(w1)}{(x)}{4.5pt}{1.5pt}{red!70!white}{red!70!white,opacity=0.2};
\qedge{(b)}{(w0)}{(c)}{4.5pt}{1.5pt}{red!70!white}{red!70!white,opacity=0.2};
\qedge{(b1)}{(w0)}{(c1)}{4.5pt}{1.5pt}{red!70!white}{red!70!white,opacity=0.2};
\qedge{(b1)}{(w1)}{(c1)}{4.5pt}{1.5pt}{red!70!white}{red!70!white,opacity=0.2};
\qedge{(x)}{(w1)}{(y)}{4.5pt}{1.5pt}{red!70!white}{red!70!white,opacity=0.2};		
	\end{tikzpicture}
	\caption {Quadruple $(w_0,b_1,c_1,w_1)$ that fulfills the conditions of Lemma~\ref{cl:1}, 
	where the \textcolor{green}{ link graph of $v$} is indicated in green and \textcolor{red}
	{hyperedges of $H$} in red.}
	\label{sfig:34}
\end{figure}

\begin{proof}
	For every $w\in N_v(b,c)\cap N_v(x,y)$ Lemma~\ref{lm:B2} states that there are at 
	least $\alpha n^2/2$ walks in $G_{vw}$ from $c$ to $x$ of length 3.
	Let 
	\[
		X_{b_1c_1}=\{ w\in N_v(b,c)\cap N_v(x,y) \colon cb_1c_1x \text{ is a walk in } G_{vw} \}
	\]
		for $b_1,c_1 \in V$. Thus
		\[
		\sum_{(b_1,c_1)\in V^2}|X_{b_1c_1}|\geq \alpha mn^2/2
	\]
		and therefore the Cauchy-Schwarz inequality yields that
		\[
		\sum_{(b_1,c_1)\in V^2}|X_{b_1c_1}|^2\geq \alpha^2 m^2n^2/4\,.
	\]
		If $b_1,c_1\in V$ and $w_0,w_1\in X_{b_1c_1}$, then $bcw_0b_1c_1w_1xy$ has the 
	desired properties.
\end{proof}

\begin{prop}\label{pr:1}
	There is an integer $K$, such that for all edges $abc, xyz\in E$ and vertices 
	$v\in N(a,b,c)\cap N(x,y,z)$ there are for some $k=k(abc,xyz)\leq K$ with 
	$k\equiv 1 \pmod {3}$ at least~$\Omega (n^k)$ many $(u_1, \ldots ,u_k)\in V^k$ for 
	which $abcu_1\ldots u_kxyz$ is 
	\begin{itemize}
		\item a walk in $H$
		\item  a squared walk in $L_v$\,.
	\end{itemize}
\end{prop}

\begin{proof}
	Recall that in Proposition~\ref{prop:zwei} we found an integer $\ell$ and a 
	function $t\colon V^{(2)}\rightarrow [\ell]$ such that for all distinct $r,s\in V$ 
	there are $\Omega(n^{t(r,s)-1})$ walks of length $t(r,s)$ from $r$ to $s$ in $G_v$. 
	By the box principle there exists an integer $t\leq \ell$ such that the set 
	$\mathcal {Q}\subseteq N_v(b,c)\times N_v(x,y)$ of  all pairs 
	$(u,u')\in N_v(b,c)\times N_v(x,y)$ with $t(u,u')=t$ satisfies
		\[
		|\mathcal{Q}|\geq \dfrac{|N_v(b,c)|\cdot |N_v(x,y)|}{\ell}
		\overset{\eqref{eq:Nv}}{\geq}
		\dfrac{n^2}{16 \ell}\,.
	\]
		For each walk $v_0v_1\ldots v_t$ in $G_v$ there are by Definition \ref{dfn:Gv} 
	at least $(\beta n^2)^t$ many $(2t)$-tuples $(b_1,c_1,\ldots ,b_t,c_t)$ such that
	\begin{enumerate}[label=\rmlabel, series=B]
		\item \label{it:3161} $b_ic_iv\in E$ for $i=1,\ldots ,t$,
		\item \label{it:3162}$v_0\in N_v(b_1,c_1)$ and $v_t\in N_v(b_t,c_t)$, 
		\item \label{it:3163}$v_i\in N_v(b_i,c_i)\cap N_v(b_{i+1},c_{i+1})$ for 
		$i=1,\ldots ,t-1$\,.
	\end{enumerate}
	
\begin{figure}[ht]
\begin{tikzpicture}[scale=1]
	
	\coordinate (b) at (-9,0.8);
	\coordinate (c) at (-8,0.8);
	\coordinate (b1) at (-6,0.8);
	\coordinate (c1) at (-5,0.8);
	\coordinate (b2) at (-3,0.8);
	\coordinate (c2) at (-2,0.8);
	
	\coordinate (w0) at (-7,3);
	\coordinate (w1) at (-4,3);
	\coordinate (d) at (-1,1.75);
	\coordinate (d1) at (-0.5,1.75);
	\coordinate (d2) at (0,1.75);

\coordinate (w3) at (1,3);
	\coordinate (w2) at (4,3);
\coordinate (bt) at (2,0.8);
	\coordinate (ct) at (3,0.8);
\coordinate (x) at (5,0.8);
	\coordinate (y) at (6,0.8);		
   
	\begin{pgfonlayer}{front}
	
	\draw[green, line width=2pt] (b) -- (c);
		\draw[green, line width=2pt] (b2) -- (w1);
		\draw[green, line width=2pt] (b2) -- (c2);
		\draw[green, line width=2pt] (w1) -- (c2);
		\draw[green, line width=2pt] (c1) -- (b1);
		\draw[green, line width=2pt] (b) -- (w0);
		\draw[green, line width=2pt] (c) -- (w0);
		\draw[green, line width=2pt] (b1) -- (w0);
		\draw[green, line width=2pt] (c1) -- (w0);
		\draw[green, line width=2pt] (b1) -- (w1);
		\draw[green, line width=2pt] (c1) -- (w1);
		\draw[green, line width=2pt] (x) -- (w2);
		\draw[green, line width=2pt] (y) -- (w2);
		\draw[green, line width=2pt] (bt) -- (w2);
		\draw[green, line width=2pt] (ct) -- (w3);
		\draw[green, line width=2pt] (bt) -- (w3);
		\draw[green, line width=2pt] (ct) -- (w2);
		\draw[green, line width=2pt] (bt) -- (ct);
				\draw[green, line width=2pt] (x) -- (y);
	
		\fill (b) circle (2pt);
		
		\fill (c) circle (2pt);
		
		\fill (b1) circle (2pt);
		
		\fill (c1) circle (2pt);
		\fill (b2) circle (2pt);
		\fill (c2) circle (2pt);
		\fill (bt) circle (2pt);
		\fill (ct) circle (2pt);
		
		\fill (w2) circle (2pt);
		\fill (w3) circle (2pt);
		\fill (x) circle (2pt);
		\fill (y) circle (2pt);
		\fill (w0) circle (2pt);
		\fill (w1) circle (2pt);
		\fill (d) circle (1pt);
		\fill (d1) circle (1pt);
		\fill (d2) circle (1pt);
		
		\node at ($(w2)+(90:0.34)$) {\footnotesize  $v_t$};
	\node at ($(w3)+(90:0.34)$) {\footnotesize $v_{t-1}$};
	\node at ($(w1)+(90:0.34)$) {\footnotesize $v_1$};
	\node at ($(w0)+(90:0.34)$) {\footnotesize $v_0$};
		\node at ($(b)+(270:0.34)$) {\footnotesize $b$};
		\node at ($(c)+(270:0.34)$) {\footnotesize $c$};
		\node at ($(b1)+(270:0.34)$) {\footnotesize $b_1$};
		\node at ($(c1)+(270:0.34)$) {\footnotesize $c_1$};
		\node at ($(b2)+(270:0.34)$) {\footnotesize $b_2$};
		\node at ($(c2)+(270:0.34)$) {\footnotesize $c_2$};
			\node at ($(bt)+(270:0.34)$) {\footnotesize $b_t$};
		\node at ($(ct)+(270:0.34)$) {\footnotesize $c_t$};
		\node at ($(x)+(270:0.34)$) {\footnotesize $x$};
				\node at ($(y)+(270:0.34)$) {\footnotesize $y$};
						
	\end{pgfonlayer}
	
	\begin{pgfonlayer}{background}
			
					\end{pgfonlayer} 

\qedge{(b)}{(w0)}{(c)}{4.5pt}{1.5pt}{red!70!white}{red!70!white,opacity=0.2};
\qedge{(b1)}{(w0)}{(c1)}{4.5pt}{1.5pt}{red!70!white}{red!70!white,opacity=0.2};
\qedge{(b1)}{(w1)}{(c1)}{4.5pt}{1.5pt}{red!70!white}{red!70!white,opacity=0.2};
\qedge{(b2)}{(w1)}{(c2)}{4.5pt}{1.5pt}{red!70!white}{red!70!white,opacity=0.2};
\qedge{(bt)}{(w2)}{(ct)}{4.5pt}{1.5pt}{red!70!white}{red!70!white,opacity=0.2};
\qedge{(bt)}{(w3)}{(ct)}{4.5pt}{1.5pt}{red!70!white}{red!70!white,opacity=0.2};
\qedge{(x)}{(w2)}{(y)}{4.5pt}{1.5pt}{red!70!white}{red!70!white,opacity=0.2};
	\end{tikzpicture}
	\caption {A $(3t+1)$-tuple $(v_0, v_1,\ldots ,v_t,b_1,c_1,\ldots ,b_t,c_t)\in V^{3t+1}$ 
	satisfying \ref{it:3161}, \ref{it:3162}, \ref{it:3163}, and \ref{it:3164}, 
	where the \textcolor{green}{ link graph of $v$} is indicated in green 
	and \textcolor{red}{hyperedges of $H$} in red.}
	\label{fig:35}
	\end{figure}
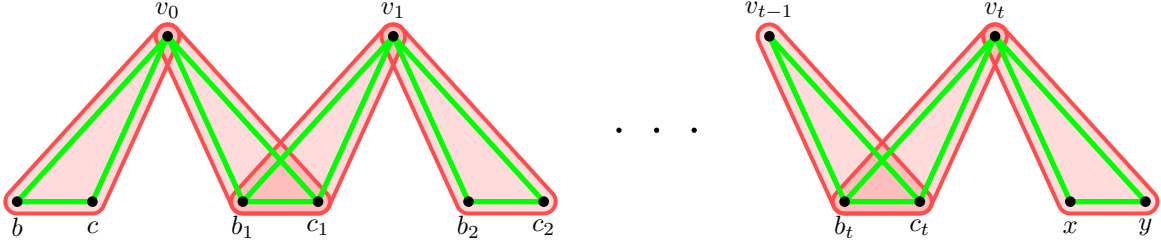

	Consequently, there are at least 
		\[
		\dfrac{n^2}{16\ell}\cdot \Omega (n^{t-1})\cdot (\beta n^2)^t=\Omega (n^{3t+1})
	\]
		$(3t+1)$-tuples $(v_0,v_1,\ldots ,v_t,b_1,c_1,\ldots ,b_t,c_t)\in V^{3t+1}$ 
	satisfying \ref{it:3161}, \ref{it:3162}, \ref{it:3163} as well as
	\begin{enumerate}[resume*=B]
		\item \label{it:3164} $v_0\in N_v(b,c)$ and $v_t\in N_v(x,y)$\,.
	\end{enumerate}
	On the other hand, we can also write the number of these $(3t+1)$-tuples as
		\[
		\sum\limits_{\seq{v} \in\Psi}
			|I_0(\seq{v})| \cdot |I_1(\seq{v})| \cdot \ldots \cdot |I_{t}(\seq{v})| \,,
	\]
		where
		\[
		\Psi =\{(b_1,c_1,\ldots , b_t,c_t)\in V^{2t}\colon  b_ic_iv\in E \text{ for } 
			i=1, \ldots ,t\}
	\]
		and for fixed $\seq{v}=(b_1,c_1,\ldots ,b_t,c_t)\in \Psi$
	\begin{itemize}
		\item  $I_0(\seq{v})=N_v(b,c)\cap N_v(b_{1},c_{1})$
		\item $I_i(\seq{v})=N_v(b_{i},c_{i})\cap N_v(b_{i+1},c_{i+1})$ 
			for $i=1,\ldots ,t-1$
		\item $I_{t}(\seq{v})=N_v(b_{t},c_t)\cap N_v(x,y)\, .$
	\end{itemize}
	Altogether we have thereby shown that
		\begin{equation}\label{eq:4}
 		\sum\limits_{\seq{v} \in\Psi}
		|I_0(\seq{v})| \cdot |I_1(\seq{v})| \cdot \ldots \cdot 
		|I_{t}(\seq{v})|\geq \Omega ( n^{3t+1})\, .
	\end{equation}
		Due to \eqref{eq:4} and Lemma \ref{cl:1} there are at least
		\begin{align*}
 		&\sum\limits_{\seq{v} \in\Psi} 
			\Omega (|I_0(\seq{v})|^{2}n^{2})
			\cdot \ldots \cdot 
			\Omega (|I_{t}(\seq{v})|^{2}n^{2})\\
 		&\geq \Omega (n^{2t+2})  \sum\limits_{\seq{v} \in\Psi} 
			(|I_0(\seq{v})|\cdot \ldots \cdot 
			|I_{t}(\seq{v})|)^2\\
 		&\geq \Omega (n^{2t+2})  \dfrac{\Big (\sum\limits_{\seq{v} \in\Psi} 
			|I_0(\seq{v})|\cdot \ldots \cdot 
				|I_{t}(\seq{v})| \Big )^2}{|\Psi|} \\
 		&\geq \Omega (n^{2t+2})\Big (\dfrac{\Omega ( n^{3t+1})}{ n^{t}} \Big ) ^{2}
			=\Omega (n^{6t+4})
	\end{align*}
		$(6t+4)$-tuples, which fulfill the conditions of 
	Proposition~\ref{pr:1}. Since $6t+4\equiv 1 \pmod {3}$ this concludes the proof.
\end{proof}

\begin{dfn}
	We call a sequence of vertices $v_1\ldots v_h$ a {\it squared $v$-walk} from $abc$ to $xyz$ 
	with $h$ {\it interior vertices} if $abcv_1\ldots v_hxyz$ is a walk in $H$ and a squared 
	walk in $L_v$.
\end{dfn}

\begin{prop}\label{pr:2}
	For all $abc, xyz\in E$ and $v\in N(a,b,c)\cap N(x,y,z)$ there are for some 
	$k'=k'(abc,xyz,v)\leq K+2$ with $k '\equiv 0 \pmod {3}$ at least $\Omega ( n^{k '})$ 
	many squared $v$-walks with~$k '$ interior vertices from $abc$ to $xyz$.
\end{prop}

\begin{proof}
	We choose vertices $d\in N_v(b,c)$ and $e\in N_v(c,d)$, and with Proposition \ref{pr:1} 
	we find at least $\Omega (n^k)$ many squared $v$-walks from $cde$ to $xyz$, where
	$k=k(cde,xyz)\leq K$ and $k\equiv 1 \pmod 3$. Notice that if $u_1\ldots u_k$ is such 
	a walk, then $deu_1\ldots u_k$ is a squared $v$-walk from $abc$ to $xyz$. 
	Since $|N_v(b,c)|,|N_v(c,d)|\geq n/4$ holds by \eqref{eq:Nv}, there are  for some 
	$k\leq K$ with $k\equiv 1 \pmod 3$ at least $\frac{n^2/16}{K}=\Omega (n^2)$ pairs 
	$(d,e)$ with $k(cde,xyz)=k$. Now altogether there are $\Omega (n^{k+2})$ 
	squared $v$-walks from $abc$ to $xyz$ with $k+2$ interior vertices.
	This implies Proposition \ref{pr:2}, since $k+2\equiv 0 \pmod 3$.
\end{proof}
	
\begin{lemma}\label{lm:Lc}
	If $abc, xyz\in E$ and $|N(a,b,c)\cap N(x,y,z)|=m$, then there is an integer 
	$t=t(abc,xyz)\leq (K+2)/3$ such that at least $\Omega (m^{t+1}n^{3t})$ squared 
	walks from $abc$ to $xyz$ with  $4t+1$ interior vertices exist.
\end{lemma}

\begin{proof}
	For every $w\in N(a,b,c)\cap N(x,y,z)$ Proposition \ref{pr:2}~states that for some 
	integer $k ' =k ' (w)\leq~K+2$ with $k ' \equiv 0 \pmod 3$ there are at least 
	$\Omega (n^{k '})$ many squared $w$-walks from $abc$ to $xyz$ with $k'$ interior vertices. 
	By the box principle there exists an 
	integer $k ''\leq K+2$ with $k '' \equiv 0 \pmod 3$  such that the set 
	$\mathcal{Q}\subseteq N(a,b,c)\cap N(x,y,z)$ of all vertices 
	$w'\in N(a,b,c)\cap N(x,y,z)$ with $k ' (w)=k ''$ satisfies
		\[
		|\mathcal{Q}|\geq \dfrac{|N(a,b,c)\cap N(x,y,z)|}{K+2 }= \dfrac{m}{K+2} \,.
	\]
	 	For $P=(u_1, \ldots , u_{k''})\in V^{k''}$ let $X_P\subseteq \mathcal{Q}$ be the 
	set of vertices $u\in {\mathcal Q}$ such that $P$ is a squared $u$-walk from $abc$ to $xyz$. 
	Since $|\mathcal{Q} |\geq m/(K+2)$, the average size of $X_P$ is at least 
	$\Omega (m/(K+2))=\Omega (m)$ by Proposition~\ref{pr:2} and double counting. 
	Since
		\[
		\dfrac{ \sum_{P\in V^{k''}} X_P^{k ''/3+1}}{n^{k''}}
		\geq 
		\Big (\dfrac{\sum_{ P\in V^{k''}} X_P}{n^{k''}}\Big )^{k ''/3+1}
		\geq 
		\Omega ( m^{k ''/3+1}) \,,
	\]
		we get
		\[
		\sum_{P\in V^{k''}} X_P^{k ''/3+1}\geq \Omega (m^{k ''/3+1} n^{k ''})\,. 
	\]
	Since $k'' \equiv 0 \pmod 3$ and every ordered $k''$-tuple $P$ of vertices 
	gives rise to $X_P^{k ''/3+1}$ squared walks from $abc$ to $xyz$ with $4k''/3+1$
	interior vertices, this implies Lemma \ref{lm:Lc} with~$t=k''/3$. 
\end{proof}

Finally we come to the main result of this section stated earlier as Proposition~\ref{p:cl}.

\begin{prop}[Connecting Lemma]\label{pr:CL}
	Given $\alpha>0$, 
	there are an integer $M$ and $\vartheta _*>0$, such that for all sufficiently large
	hypergraphs $H=(V,E)$ with $\delta_2(H)\geq (4/5+\alpha)|V|$ and all disjoint triples 
	$(a,b,c)$ and $(x,y,z)$ with $abc,xyz\in E$ there exists some $m<M$ for which there are at 
	least $\vartheta _* n^m$ squared paths from $abc$ to $xyz$ with $m$ internal vertices.
\end{prop}

\begin{proof}
	Recall that in Proposition \ref{prop:033} we found an integer $\ell$ and a function 
	$t\colon V^{(2)}\rightarrow [\ell]$ such that for all distinct $r,s\in V$ there are 
	$\Omega(n^{t(r,s)-1})$ walks of length $t(r,s)$ from $r$ to $s$ in~$G_3$. 
	By the box principle there exists an integer $t\leq \ell$ such that the set 
	$\mathcal{Q}\subseteq N(a,b,c)\times N(x,y,z)$ of pairs 
	$(u,u')\in N(a,b,c)\times N(x,y,z)$ with $t(u,u')=t$ satisfies 
		\[
		|\mathcal{Q}|\geq \dfrac{|N(a,b,c)|\cdot |N(x,y,z)|}{\ell}\geq \dfrac{n^2}{16 \ell}\,.
	\]
		For each walk $v_0v_1\ldots v_t$ in $G_3$ there are by Definition~\ref{dfn:G3} at least 
	$(\beta n^3)^t$ many $(3t)$-tuples $(a_1,b_1,c_1,\ldots ,a_t ,b_t,c_t)$ such that
	\begin{enumerate}[label=\rmlabel, series=C]
		\item \label{it:3201} $a_ib_ic_i\in E$ for $i=1,\ldots ,t$
		\item \label{it:3202}$v_0\in N(a_1,b_1,c_1)$ and $v_t\in N(a_t,b_t,c_t)$ 
		\item \label{it:3203}$v_i\in N(a_i,b_i,c_i)\cap N(a_{i+1},b_{i+1},c_{i+1})$ 
			for $i=1,\ldots ,t-1$\,.
	\end{enumerate}
	Consequently, there are at least 
		\[
		\dfrac{n^2}{16\ell}\cdot \Omega (n^{t-1}) \cdot (\beta n^3)^t=\Omega (n^{4t+1})
	\]
		$(4t+1)$-tuples $(v_0,\ldots ,v_t,a_1,b_1,c_1,\ldots ,a_t ,b_t,c_t)\in V^{4t+1}$ 
	satisfying~\ref{it:3201},~\ref{it:3202},~\ref{it:3203} as well as
	\begin{enumerate}[resume*=C]
		\item \label{it:3204} $v_0\in N(a,b,c)$ and $v_t\in N(x,y,z)$\,.
	\end{enumerate}
		
\begin{figure}[ht]
\begin{tikzpicture}[scale=1]
	
	\coordinate (a) at (-10,0.8);
	\coordinate (b) at (-9,0);
	\coordinate (c) at (-8,0.8);
	\coordinate (a1) at (-7,0.8);
	\coordinate (b1) at (-6,0);
	\coordinate (c1) at (-5,0.8);
	\coordinate (a2) at (-4,0.8);
	\coordinate (b2) at (-3,0);
	\coordinate (c2) at (-2,0.8);
	
	\coordinate (w0) at (-7.5,3);
	\coordinate (w1) at (-4.5,3);
	\coordinate (d) at (-1.5,1.75);
	\coordinate (d1) at (-1,1.75);
	\coordinate (d2) at (-0.5,1.75);

\coordinate (w3) at (0.5,3);
	\coordinate (w2) at (3.5,3);
	\coordinate (at) at (1,0.8);
\coordinate (bt) at (2,0);
	\coordinate (ct) at (3,0.8);
\coordinate (x) at (4,0.8);
	\coordinate (y) at (5,0);	
	\coordinate (z) at (6,0.8);			
   
	\begin{pgfonlayer}{front}
	
\draw[green, line width=2pt] (b) -- (c);
		\draw[green, line width=2pt] (b2) -- (a2);
		\draw[green, line width=2pt] (b2) -- (c2);
						\draw[green, line width=2pt] (a) -- (b);
		\draw[green, line width=2pt] (a1) -- (b1);
		\draw[green, line width=2pt] (c1) -- (b1);
		
												\draw[green, line width=2pt] (bt) -- (at);
		\draw[green, line width=2pt] (bt) -- (ct);
		\draw[green, line width=2pt] (y) -- (z);
		\draw[green, line width=2pt] (x) -- (y);
		\draw[green, line width=2pt] (a) -- (c);
		\draw[green, line width=2pt] (a1) -- (c1);
		\draw[green, line width=2pt] (a2) -- (c2);
		\draw[green, line width=2pt] (at) -- (ct);
		\draw[green, line width=2pt] (x) -- (z);

		\fill (a) circle (2pt);		
		\fill (b) circle (2pt);
		\fill (c) circle (2pt);
		\fill (a1) circle (2pt);
		\fill (b1) circle (2pt);
		
		\fill (c1) circle (2pt);
		\fill (a2) circle (2pt);
		\fill (b2) circle (2pt);
		\fill (c2) circle (2pt);
		\fill (at) circle (2pt);
		\fill (bt) circle (2pt);
		\fill (ct) circle (2pt);
		
		\fill (w2) circle (2pt);
		\fill (w3) circle (2pt);
		\fill (x) circle (2pt);
		\fill (y) circle (2pt);
		\fill (z) circle (2pt);
		\fill (w0) circle (2pt);
		\fill (w1) circle (2pt);
		\fill (d) circle (1pt);
		\fill (d1) circle (1pt);
		\fill (d2) circle (1pt);
		
		\node at ($(w2)+(90:0.36)$) {\footnotesize  $v_{t}$};
	\node at ($(w3)+(90:0.34)$) {\footnotesize $v_{t-1}$};
	\node at ($(w1)+(90:0.36)$) {\footnotesize $v_1$};
	\node at ($(w0)+(90:0.36)$) {\footnotesize $v_0$};
	\node at ($(a)+(270:0.34)$) {\footnotesize $a$};
		\node at ($(b)+(270:0.34)$) {\footnotesize $b$};
		\node at ($(c)+(270:0.34)$) {\footnotesize $c$};
		\node at ($(a1)+(270:0.34)$) {\footnotesize $a_1$};
		\node at ($(b1)+(270:0.34)$) {\footnotesize $b_1$};
		\node at ($(c1)+(270:0.34)$) {\footnotesize $c_1$};
		\node at ($(a2)+(270:0.34)$) {\footnotesize $a_2$};
		\node at ($(b2)+(270:0.34)$) {\footnotesize $b_2$};
		\node at ($(c2)+(270:0.34)$) {\footnotesize $c_2$};
			\node at ($(at)+(270:0.34)$) {\footnotesize $a_t$};
			\node at ($(bt)+(270:0.34)$) {\footnotesize $b_t$};
		\node at ($(ct)+(270:0.34)$) {\footnotesize $c_t$};
		\node at ($(x)+(270:0.34)$) {\footnotesize $x$};
				\node at ($(y)+(270:0.34)$) {\footnotesize $y$};
		\node at ($(z)+(270:0.34)$) {\footnotesize $z$};
						
	\end{pgfonlayer}
	
	\begin{pgfonlayer}{background}
	
	\end{pgfonlayer} 

\qedge{(b)}{(a)}{(c)}{4.5pt}{1.5pt}{red!70!white}{red!70!white,opacity=0.2};
\qedge{(b1)}{(a1)}{(c1)}{4.5pt}{1.5pt}{red!70!white}{red!70!white,opacity=0.2};
\qedge{(b2)}{(a2)}{(c2)}{4.5pt}{1.5pt}{red!70!white}{red!70!white,opacity=0.2};
\qedge{(bt)}{(at)}{(ct)}{4.5pt}{1.5pt}{red!70!white}{red!70!white,opacity=0.2};
\qedge{(y)}{(x)}{(z)}{4.5pt}{1.5pt}{red!70!white}{red!70!white,opacity=0.2};
\redge{(a1)}{(w0)}{(c1)}{(b1)}{6pt}{1.5pt}{orange!70!white}{orange!70!white,opacity=0.2};
\redge{(a1)}{(w1)}{(c1)}{(b1)}{6pt}{1.5pt}{orange!70!white}{orange!70!white,opacity=0.2};
\redge{(a)}{(w0)}{(c)}{(b)}{6pt}{1.5pt}{orange!70!white}{orange!70!white,opacity=0.2};
\redge{(a2)}{(w1)}{(c2)}{(b2)}{6pt}{1.5pt}{orange!70!white}{orange!70!white,opacity=0.2};
\redge{(at)}{(w3)}{(ct)}{(bt)}{6pt}{1.5pt}{orange!70!white}{orange!70!white,opacity=0.2};
\redge{(at)}{(w2)}{(ct)}{(bt)}{6pt}{1.5pt}{orange!70!white}{orange!70!white,opacity=0.2};
\redge{(x)}{(w2)}{(z)}{(y)}{6pt}{1.5pt}{orange!70!white}{orange!70!white,opacity=0.2};

	\end{tikzpicture}
	\caption {A $(4t+1)$-tuple 
	$(v_0,\ldots ,v_t,a_1,b_1,c_1,\ldots ,a_t,b_t,c_t)\in V^{4t+1}$ 
	satisfying~\ref{it:3201},~\ref{it:3202},~\ref{it:3203},~and \ref{it:3204}, 
	where orange quadruples indicate a \textcolor{orange}{copy of~$K_4^{(3)}$}, 
	\textcolor{red}{hyperedges of $H$ }are indicated in red, and green pairs are 
	in the \textcolor{green}{ link graph }of the corresponding $v_i$.}
	\label{fig:36}
	\end{figure}
	
	On the other hand, we can also write the number of these $(4t+1)$-tuples as
		\[
		 \sum\limits_{\seq{v} \in\Psi}
		 |I_0(\seq{v})| \cdot |I_1(\seq{v})| \cdot \ldots \cdot |I_{t}(\seq{v})| \,,
	\]
		where
		\[
		\Psi =\{(a_1,b_1,c_1,\ldots , a_t,b_t,c_t)\in V^{3t}\colon 
			a_ib_ic_i\in E \text{ for } i=1,\ldots, t \}
	\]
		and for fixed
	$\seq{v}=(a_1,b_1,c_1,\ldots ,a_t,b_t,c_t)\in \Psi$
	\begin{itemize}
		\item  $I_0(\seq{v})=N(a,b,c)\cap N(a_1,b_1,c_1)$
		\item $I_i(\seq{v})=N(a_{i},b_{i},c_{i})\cap N(a_{i+1},b_{i+1},c_{i+1})$ 
			for $i=1,\ldots ,t-1$
		\item $I_{t}(\seq{v})=N(a_t,b_{t},c_t)\cap N(x,y,z)$
	\end{itemize}
	Altogether we have thereby shown that
		\begin{equation*}\label{eq:6}
 		\sum\limits_{\seq{v} \in\Psi}
		|I_0(\seq{v})| \cdot |I_1(\seq{v})| 
		\cdot \ldots \cdot |I_{t}(\seq{v})|
		\geq 
		\Omega (n^{4t+1}) \,.
	\end{equation*}
		Lemma~\ref{lm:Lc} gives us for every $\seq{v} \in\Psi$ some integers 
	\begin{itemize}
		\item $t_0(\seq{v})=t(abc, a_1b_1c_1)$
		\item $t_i(\seq{v})=t(a_ib_ic_i, a_{i+1}b_{i+1}c_{i+1})$ for $i=1, 2, \ldots, t-1$
		\item and $t_t(\seq{v})=t(a_tb_tc_t, xyz)$.
	\end{itemize}
	By the box principle there are $\Psi_\star\subseteq \Psi$ and a $(t+1)$-tuple 
	$(t_0, \ldots, t_t)\in [1, (K+2)/3]^{t+1}$ such that 
	\begin{equation}\label{eq:666}
 		\sum\limits_{\seq{v} \in\Psi_\star}
		|I_0(\seq{v})| \cdot |I_1(\seq{v})| 
		\cdot \ldots \cdot |I_{t}(\seq{v})|
		\geq 
		\Omega (n^{4t+1}) 
	\end{equation}
	and $t_i(\seq{v})=t_i$ for all $i\in\{0, \ldots, t\}$ and $\seq{v}\in\Psi_*$.
	Set  $m=4t+4\sum_{i=0}^{t} t_i+1$.
	Due to Lemma~\ref{lm:Lc} there are at least
		\begin{align*}
 		&\sum\limits_{\seq{v} \in\Psi_\star} 
		\Omega (|I_0(\seq{v})|^{t_0+1}n^{3t_0})
		\cdot \ldots \cdot 
		\Omega (|I_{t}(\seq{v})|^{t_{t}+1}n^{3t_{t}})\\
 		&=\Omega (n^{3\sum_{i=0}^{t} t_i})  
			\sum\limits_{\seq{v} \in\Psi_\star
			}
			|I_0(\seq{v})|^{t_0+1}\cdot \ldots \cdot 
			|I_{t}(\seq{v})|^{t_{t}+1}
	\end{align*}
	 	$m$-tuples, which up to repeated vertices fulfill the conditions of 
		Proposition~\ref{pr:CL}. Let $T=\max(t_0,\ldots ,t_{t})$. Since 
		\[
		|I_i(\seq{v})|^{T+1}
		=
		|I_i(\seq{v})|^{t_i+1} \cdot 
			|I_i(\seq{v})|^{T-t_i}
		\leq 
		|I_i(\seq{v})|^{t_i+1} \cdot n^{T-t_i} \,,
	\]
		we get 
		\begin{align*}
		n^{T(t+1)-\sum_{i=0}^{t} t_i} & \sum\limits_{\seq{v}\in \Psi_\star} 
			\prod _{i=0}^{t}  |I_i(\seq{v})|^{t_i+1} 
		=\sum\limits_{\seq{v}\in \Psi_\star} \prod _{i=0}^{t}  n^{T-t_i}
			|I_i(\seq{v})|^{t_i+1} 
 		\geq \sum\limits_{\seq{v} \in\Psi_\star}  
			|I_0(\seq{v})|^{T+1}\cdot \ldots \cdot 
			|I_{t}(\seq{v})|^{T+1}\\
  		&=\sum\limits_{\seq{v} \in\Psi_\star} 
			(|I_0(\seq{v})| \cdot \ldots \cdot 
			|I_{t}(\seq{v})|)^{T+1}
		\geq \Bigg ( \dfrac {\sum\limits_{\seq{v} \in\Psi_\star}  
			|I_0(\seq{v})|\cdot \ldots \cdot 
			|I_{t}(\seq{v})| }{|\Psi_\star |}\Bigg )^{T+1} \cdot |\Psi_\star|\\
 		&\overset{\eqref{eq:666}}{\geq} \Big (\dfrac{\Omega  (n^{4t+1})}{ n^{3t}} \Big ) ^{T+1}\cdot  n^{3t}
		\geq \Omega  (n^{3t+(t+1)(T+1)}) \,,
	\end{align*}
		which implies that
		\begin{align*}
		&\Omega (n^{3\sum_{i=0}^{t} t_i}) \sum\limits_{\seq{v} \in\Psi_\star}
			|I_0(\seq{v})|^{t_0+1}
			\cdot \ldots \cdot 
			|I_{t}(\seq{v})|^{t_{t}+1}\\
	&\geq \Omega ( n^{3t+(t+1)+\sum_{i=0}^{t} t_i+3\sum_{i=0}^{t} t_i})=\Omega (n^{m}) \,.
	\end{align*}
		At most $O(n^{m-1})$ tuples can fail being paths due to repeated vertices, thus there are 
	$\Omega(n^m)$ squared paths from $abc$ to $xyz$. This proves Proposition \ref{pr:CL} with 
	$M=\lceil 4\ell + 4(\ell+1)\cdot \frac{K+2}{3}+2\rceil$, since 
	$m=4t+4\sum_{i=0}^{t} t_i+1\leq 4\ell + 4(\ell+1)\cdot \frac{K+2}{3}+1$.
\end{proof}

\section{Reservoir Set} \label{sec:res} 

Our treatment of the reservoir set follows closely the approach of \cite{R3S2}. The 
setup discussed in this section is that we have 
\begin{enumerate}
	\item[$\bullet$] $1\gg \alpha\gg M^{-1} \gg \theta_* \gg n^{-1}$ such that the conclusion 
		of the connecting lemma holds, 
	\item[$\bullet$] and a hypergraph $H=(V, E)$ with $|V|=n$ 
		and $\delta_2(H)\ge \bigl(\frac45+\alpha\bigr)n$.	
\end{enumerate}

\begin{prop}\label{prop:r41}
	There exists a reservoir set $\mathcal {R} \subseteq V$ with 
	$\frac{\vartheta_*^2 n}{2}\leq |\mathcal {R}|\leq \vartheta_*^2 n$, 
	such that for all disjoint triples $(a,b,c)$ and $(x,y,z)$ with $abc,xyz\in E$ 
	there exists $m<M$ such that there are at least $\vartheta_{*}|\cR|^m/2$ connecting 
	squared paths in $H$ all of whose $m$ internal vertices belong to $\mathcal{R}$.
\end{prop}

\begin{proof}
	Consider a random subset $\mathcal{R} \subseteq V$ with elements included independently 
	with probability
		\[
		p=\Big(1-\dfrac{3}{10M}\Big)\vartheta_*^2 \,.
	\]
		Therefore $|\mathcal{R}|$ is binomially distributed and Chernoff's inequality yields
		\begin{equation}\label{eq:7}
		\PP(|\mathcal{R} |< \vartheta_*^2 n/2 )=o(1)\, .
	\end{equation}
		Since 
		\begin{equation*}		\vartheta_*^2 n \geq (4/3)^{1/M} pn\geq (1+c) \EE[|\mathcal{R} |] 
	\end{equation*}
		for some sufficiently small $c=c(M)>0$, we have
		\begin{equation}\label{eq:8}
		\PP(|\mathcal{R} |> \vartheta_*^2 n )
		 \leq 
		 \PP\big(|\mathcal{R} |> (4/3)^{1/M} pn \big ) =o(1) \,.
	\end{equation}
	 	
	The connecting lemma ensures that for all triples $(a,b,c)$ and $(x,y,z)$ with $avc, xyz\in E$
	there are at least~$\vartheta_{*}n^m$ squared paths connecting them with $m=m(abc,xyz)<M$ internal
	vertices. 
 
 	Let $X=X((a,b,c),(x,y,z))$ be the random variable counting the number of squared paths 
	from $(a,b,c)$ to $(x,y,z)$ with $m$ internal vertices in $\mathcal{R}$. 
	We get
 		\begin{equation}\label{eq:9}
		\EE [X]\geq p^m \vartheta_{*}n^m\, .
 	\end{equation}
		Including or not including a particular vertex into $\mathcal{R}$ affects the 
	random variable $X$ by at most $mn^{m-1}$, wherefore the Azuma-Hoeffding 
	inequality (see, e.g.,~\cite{JLR00}*{Corollary~2.27}) implies
		\begin{align} \label{eq:10}
		\PP\big(X\leq \tfrac{2}{3} \vartheta_{*}(pn)^m \big)&
			\overset {\eqref{eq:9}}{\leq}  
			\PP\big (X\leq \tfrac{2}{3}\EE [X]\big )\nonumber \\
		&~\leq \exp \Bigg(-\frac{2\EE [X]^2}{9n(mn^{m-1})^2}\Bigg)=\exp (-\Omega (n)) \,.  
	\end{align}
		Since there are at most $n^6$ pairs of triples that we have to consider, the union bound 
	and \eqref{eq:7}, \eqref{eq:8} tell us that asymptotically almost surely the reservoir 
	$\mathcal{R}$ satisfies 
		\begin{equation}\label{equ:46}
		\dfrac{\vartheta_*^2 n}{2}\leq |\mathcal{R}|
		\leq (4/3)^{1/M}pn
		\leq \vartheta_*^2 n 
	\end{equation}
		and
 		\begin{equation} \label{equ:47}
		X((a,b,c),(x,y,z))\geq  \dfrac{2}{3} \vartheta_{*}(pn)^m 
	\end{equation}
		for all pairs of disjoint edges $abc,xyz\in E$.
	In particular, there is some $\mathcal{R}\subseteq V$ satisfying \eqref{equ:46} 
	and \eqref{equ:47}.
	Now it follows that
		\[
		X((a,b,c),(x,y,z))
		\geq  
		\frac 23 \theta_* \bigl((4/3)^{-1/M}|\cR|\bigr)^m
		\ge
		\vartheta_{*} |\mathcal{R}|^m/2
	\]
	  	holds for all disjoint $abc,xyz\in E$ as well, meaning that $\mathcal{R}$ has the desired 
	properties.
\end{proof}

\begin{lemma}\label{lm:res42}
	Let $\mathcal{R}\subseteq V$ be a reservoir set as given by Proposition~\ref{prop:r41}
	and let $\mathcal{R'}\subseteq \mathcal{R}$ be an arbitrary subset 
	of size at most $\vartheta_{*}^4n$. Then for all disjoint triples $(a,b,c)$ and $(x,y,z)$ 
	with $abc, xyz\in E$ there exist for some $m<M$ a connecting squared path with $m$ 
	internal vertices in $H$ 
	whose internal vertices belong to $\mathcal{R}\setminus \mathcal{R'}$.
\end{lemma}

\begin{proof}
	Let $m<M$ be such that there are $\theta_*|\cR|^m/2$ squared path from $(a, b, c)$ to 
	$(x, y, z)$ with $m$ internal vertices all of which belong to $\cR$. 
	Since $|\mathcal{R}|\geq \dfrac{\vartheta_*^2 n}{2}$ and 
	$\vartheta_*\ll M^{-1}$, we can arrange that 
		\[
		|\mathcal{R'}|\leq \vartheta_{*}^4n\leq \dfrac {\vartheta_{*}}{4m}|\mathcal{R}| \,.
	\]
		Every vertex in $\mathcal{R'}$ is a member of at most $m|\mathcal{R}|^{m-1}$ squared 
	paths with internal vertices in $\mathcal{R}$. Consequently, there are at least
		\[
		\dfrac{\vartheta_{*}}{2}|\mathcal{R}|^m-|\mathcal{R'}| m |\mathcal{R}|^{m-1}
		\geq 
		\dfrac{\vartheta_{*}}{2}|\mathcal{R}|^m- \dfrac {\vartheta_{*}}{4m}m |\mathcal{R}|^{m}>0
	\]
		such squared paths with all internal vertices in $\mathcal{R}\setminus \mathcal{R'}$.
\end{proof}

To conclude this section we remark that taken together Proposition~\ref{prop:r41} and 
Lemma~\ref{lm:res42} entail Proposition~\ref{p:res}.

\section{Absorbing Path} \label{sec:abs}
The goal of this section is to establish Proposition \ref{p:abs} which, let us recall, 
requires the minimum degree condition $\delta_2(H)\geq (4/5+\alpha) |V(H)|$. 
The common assumptions of all statements of this section are that we have
\begin{itemize}
	\item $1\gg \alpha \gg M^{-1}\gg \vartheta_* \gg n^{-1}$ such that the conclusion of the 
		connecting lemma holds,
	\item a hypergraph $H=(V,E)$ with $|V|=n$ and $\delta_2(H)\geq (4/5+\alpha) n$,
	\item and a reservoir set $\mathcal{R}\subseteq V$ satisfying, in particular, 
		that $|\mathcal{R}|\leq \vartheta_*^2n$.
\end{itemize}

\begin{dfn}\label{dfn:abs}
	Given a vertex $v\in V$ we say that a $6$-tuple $(a,b,c,d,e,f)\in (V\setminus \{ v\})^6$ 
	of distinct vertices is a \it{$v$-absorber} if
	$abcdef$ and $abcvdef$ are squared paths in $H$.
\end{dfn}
	
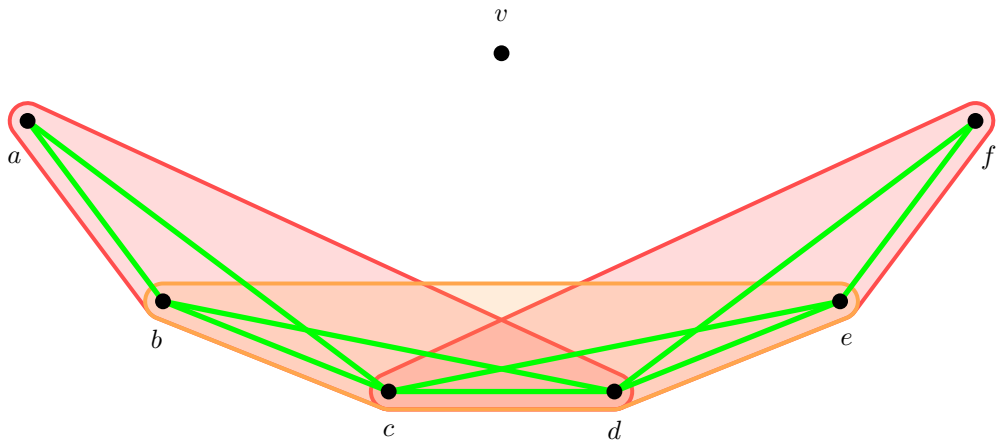
\begin{figure}[ht]
\begin{tikzpicture}[scale=1.5]
	
	\coordinate (a) at (-5.2,2.4);
	\coordinate (b) at (-4,0.8);
	\coordinate (c) at (-2,0);
	\coordinate (d) at (0,0);
	\coordinate (e) at (2,0.8);
	\coordinate (f) at (3.2,2.4);
	\coordinate (v) at (-1,3);

	\begin{pgfonlayer}{front}
	\draw[green, line width=2pt] (a) -- (b);
		\draw[green, line width=2pt] (b) -- (c);
														\draw[green, line width=2pt] (c) -- (d);
		\draw[green, line width=2pt] (e) -- (d);
		\draw[green, line width=2pt] (f) -- (e);
		\draw[green, line width=2pt] (a) -- (c);
		\draw[green, line width=2pt] (b) -- (d);
		\draw[green, line width=2pt] (c) -- (e);
		\draw[green, line width=2pt] (d) -- (f);
	
		\fill (a) circle (2pt);
		
		\fill (b) circle (2pt);
		
		\fill (c) circle (2pt);
		
		\fill (d) circle (2pt);
		
		\fill (e) circle (2pt);
		\fill (f) circle (2pt);
		\fill (v) circle (2pt);

	\node at ($(d)+(270:0.34)$) {\footnotesize $d$};
	\node at ($(v)+(90:0.34)$) {\footnotesize $v$};
		\node at ($(a)+(250:0.34)$) {\footnotesize $a$};
		\node at ($(b)+(260:0.34)$) {\footnotesize $b$};
		\node at ($(c)+(270:0.34)$) {\footnotesize $c$};
			\node at ($(e)+(280:0.34)$) {\footnotesize $e$};
				\node at ($(f)+(290:0.34)$) {\footnotesize $f$};
						
	\end{pgfonlayer}

\redge{(a)}{(d)}{(c)}{(b)}{4.5pt}{1.5pt}{red!70!white}{red!70!white,opacity=0.2};
\redge{(c)}{(f)}{(e)}{(d)} {4.5pt}{1.5pt}{red!70!white}{red!70!white,opacity=0.2};
\redge{(b)}{(e)}{(d)}{(c)}{4.5pt}{1.5pt}{orange!70!white}{orange!70!white,opacity=0.2};

	\end{tikzpicture}
	\caption {Example of a $v$-absorber, where the \textcolor{green}{link graph of $v$} is indicated in green and orange or red $4$-edges indicate a \textcolor{red}{copy of $K_4^{(3)}$}.}
	\label{fig:absorber}
	\end{figure}

Let us note that if a squared path $P$ contains some $v$-absorber $abcdef$ as a subpath 
and $v\not\in V(P)$, then we may insert $v$ into $P$ between $c$ and $d$ thus obtaining another
squared path $P'$ having the same end-triples as $P$ and with $V(P')=V(P)\cup\{v\}$. 
When using this argument, we say that we apply the {\it absorbing property} of $abcdef$.

\begin{lemma}\label{lm:vabsorb}
	For every $v\in V$ there are at least $\alpha^3 n^6$ many $v$-absorbers in 
	$(V\setminus \mathcal{R}) ^6$.
\end{lemma}

\begin{proof}
	Given $v\in V$ we choose the vertices of the 6-tuple in alphabetic order. 
	For the first vertex we have $n$ possible choices and for the second we still 
	have more than $4n/5 $ possibilities, since we only have the condition that $vab\in E$. 
	For the third vertex we already have 3 conditions, namely $abc, vbc, vac\in E$. 
	Consequently, we have more than $2n/5$ choices for $c$. For the vertices $d, e, f$ 
	we always have 5 conditions, so we have for each of them at least~$5\alpha n$ 
	possible choices. This implies that for given $v\in V$ we find more than 
		\[
		n\cdot 4n/5 \cdot 2n/5 \cdot (5\alpha n)^3 = 40\alpha^3n^6 
	\]
		6-tuples meeting all the requirements from the $v$-absorber definition except that 
	some of the 7 vertices $v, a, \ldots, f$ might coincide. There are at most 
	$\binom {7}{2}n^5=21n^5$ such bad 6-tuples and at most $6\vartheta_*^2n^6$ 
	members of $V^6$ can use a vertex from the reservoir. Consequently, the number of 
	$v$-absorbers in $(V\setminus \mathcal{R})^6$ is at least 
	$\bigl(40\alpha ^3-\frac{21}{n}-6\vartheta_*^2\bigr)n^6\geq \alpha^3 n^6$.
\end{proof}

\begin{lemma}\label{lm:F}
	There is a set $\mathcal{F}\subseteq (V\setminus \mathcal{R})^6$ with the 
	following properties:
	\begin{enumerate}
		\item \label{it:531} $|\mathcal{F}|\leq 8 \alpha ^{-3}\vartheta_*^2n$,
		\item \label{it:532} all vertices of every 6-tuple in $\mathcal{F}$ are 
			distinct and the 6-tuples in $\mathcal{F}$ are pairwise disjoint,
		\item \label{it:533} if $(a,b,c,d,e,f)\in \mathcal{F}$, then $abcdef$ is a squared 
		path in $H$
				\item \label{it:534} and for every $v\in V$ there are at least $2\vartheta_*^2n$ 
		many $v$-absorbers in $\mathcal{F}$.
	\end{enumerate}
\end{lemma}

\begin{proof}
	Consider a random selection $\mathcal{X}\subseteq (V\setminus \mathcal{R})^6$ containing 
	each 6-tuple independently with probability $p=\gamma n^{-5}$, where 
	$\gamma = 4\vartheta_*^2/\alpha^3$. Since $\EE[|\mathcal{X}|]\leq pn^6=\gamma n$, 
	Markov's inequality yields
		\begin{equation}\label{eq:11}
		\PP(|\mathcal{X}|>2\gamma n)\leq 1/2\, .
	\end{equation}
		
	We call two distinct 6-tuples from $V^6$ {\it overlapping} if there is a vertex 
	occurring in both. There are at most $36n^{11}$ ordered pairs of overlapping 6-tuples. 
	Let $P$ be the random variable giving the number of such pairs both of whose components 
	are in $\mathcal{X}$. Since $\EE[P]\leq 36n^{11}p^2=36\gamma^2 n $ and 
	$12\gamma \leq \vartheta_*$, Markov's inequality yields
		\begin{equation}\label{eq:12}
		\PP (P> \vartheta_*^2n)\leq \PP(P>144\gamma^2n)\leq \dfrac{1}{4} \,	.
	\end{equation}
	
	In view of Lemma \ref{lm:vabsorb} for each vertex $v\in V$ the set $A_v$ containing 
	all $v$-absorbers in $(V\setminus \mathcal{R})^6$ has the property 
	$\EE[|A_v\cap \mathcal{X}|]\geq \alpha^3 n^6p=\alpha^3 \gamma n=4\vartheta_*^2n$.
	Since $|A_v\cap \mathcal{X}|$ is binomially distributed, Chernoff's inequality gives for 
	every $v\in V$
		\begin{equation}\label{eq:13}
		\PP(|A_v\cap \mathcal{X}|\leq 3\vartheta_*^2n)
		\leq 
		\exp (-\Omega(n))<\dfrac{1}{5n} \,.
	\end{equation}
	
	Owing to \eqref{eq:11}, \eqref{eq:12}, and \eqref{eq:13} there is an ``instance'' 
	$\mathcal{F}_\star$ of $\mathcal{X}$ satisfying the following:
	\begin{itemize}
		\item $|\mathcal{F}_\star|\leq 2\gamma n$,
		\item $\mathcal{F}_\star$ contains at most $\vartheta_*^2n$ ordered pairs of 
			overlapping $6$-tuples,
		\item and for every $v\in V$ the number of $v$-absorbers in $\mathcal{F}_\star$ 
			is at least $3\vartheta_*^2n$.
	\end{itemize}

	If we delete from $\mathcal{F}_\star$ all the 6-tuples containing some vertex more 
	than once, all that belong to an overlapping pair, and all violating \eqref{it:533}, 
	we get a set $\mathcal{F}$ which fulfills \eqref{it:531}, since 
	$|\mathcal{F}|\leq |\mathcal{F}_\star|$. The properties~\eqref{it:532} and~\eqref{it:533}
	hold by construction and for \eqref{it:534} we recall that $v$-absorbers 
	satisfy~\eqref{it:533} by definition. Therefore the set $\mathcal{F}$ has all the 
	desired properties. 
\end{proof}

We are now ready to prove Proposition~\ref{p:abs}, which we restate for the reader's convenience.

\begin{prop}[Absorbing path] \label{pr:abs}
	Let $1\gg \alpha \gg 1/M \gg \vartheta_*$ be such that the conclusion of the 
	connecting lemma holds, let $H=(V,E)$ be a sufficiently 
	large hypergraph with $|V|=n$ and $\delta_2(H)\geq (4/5+\alpha)n$,
	and let $\cR\subseteq V$ be a reservoir set as provided by Proposition~\ref{p:res}.
	There exists an (absorbing) squared path $P_A\subseteq H-\mathcal{R}$ such that 
	\begin{enumerate}
	\item\label{it:541}  $|V(P_A)|\leq \vartheta_*n$,
	\item\label{it:542}  and for every set $X\subseteq V\setminus V(P_A)$ 
		with $|X|\leq 2\vartheta_*^2n$ there 
		is a squared path in $H$ whose set of vertices is $V(P_A)\cup X$ and whose end-triples
		are the same as those of~$P_A$.
	\end{enumerate}
\end{prop}

\begin{proof}
	Let $\mathcal{F}\subseteq (V\setminus \mathcal{R})^6$ be as obtained in 
	Lemma~\ref{lm:F}. Recall that $\mathcal{F}$ is a family of at most 
	$8\alpha^{-3}\vartheta_*^2n$ vertex-disjoint squared paths with six vertices.

	We will prove that there is a path $P_A\subseteq H-\mathcal{R}$ with the 
	following properties:
	\begin{enumerate}[label=(\alph*)]
		\item \label{it:54a} $P_A$ contains all members of $\mathcal{F}$ as subpaths, 
		\item \label{it:54b} $|V(P_A)|\leq (M+6)|\mathcal{F}|$.
	\end{enumerate}
	Basically we will construct such a path $P_A$ starting with any member of $\mathcal{F}$ 
	by applying the connecting lemma $|\mathcal{F}|-1$ times, attaching one further part 
	from $\mathcal{F}$ each time. 

	Let $\mathcal{F}_*\subseteq \mathcal{F}$ be a maximal subset such that some path 
	$P_A^*\subseteq H-\mathcal{R}$ has the properties~\ref{it:54a} and~\ref{it:54b} 
	with $\mathcal{F}$ replaced by $\mathcal{F_*}$. Obviously $P_A^*\neq \emptyset $. 
	From \ref{it:54b} and $1\gg \alpha, M^{-1} \gg \vartheta_*$ we infer
		\begin{equation}\label{eq:14}
		|V(P_A^*)|\leq (M+6)|\mathcal{F}_*|\leq 2M|\mathcal{F}|
		\leq 
		16M\alpha ^{-3}\vartheta_*^2n\leq \vartheta_*^{3/2}n
	\end{equation}
		and thus our upper bound on the size of the reservoir leads to
		\begin{equation}\label{eq:15}
		|V(P_A^*)|+|\mathcal{R}|\leq 2\vartheta_*^{3/2}n\leq \dfrac{\vartheta_* n}{2M} \,.
	\end{equation}
		Assume for the sake of contradiction that $\mathcal{F}_*\neq \mathcal{F}$. 
	Let $(x,y,z)$ be the ending triple of~$P_A^*$ and let $P$ be an arbitrary path 
	in $\mathcal{F}\setminus \mathcal{F}_*$ with starting triple $(u,v,w)$. 
	Then the connecting lemma tells us that there are at least $\vartheta_*n^m$ 
	connecting squared paths with $m$ interior vertices, where $m=m(xyz,uvw)<M$. 
	By \eqref{eq:15} at least half of them are disjoint to $V(P_A^*)\cup \mathcal{R}$. 
	Any such connection gives us a path $P_A^{**}\subseteq H-\mathcal{R}$ starting 
	with~$P_A^*$, ending with $P$ and satisfying
		\[
		|V(P_A^{**})|=|V(P_A^*)|+m+|V(P)|\leq |V(P_A^*)|+ m+6\leq (M+6)(|\mathcal{F}_*|+1) \,. 
	\]
		So $\mathcal{F}_*\cup \{P\}$ contradicts the maximality of $\mathcal{F}_*$ and proves 
	that we have indeed $\mathcal{F}_*=\mathcal{F}$. Therefore there exists a path $P_A$ 
	with the properties \ref{it:54a} and \ref{it:54b}.
	
	As proved in \eqref{eq:14} this path satisfies condition \eqref{it:541} of 
	Proposition~\ref{pr:abs}. To establish \eqref{it:542} one absorbs the up to at 
	most $2\vartheta_*^2n$ vertices in $X$ one by one into $P_A$. This is possible 
	due to \ref{it:54a} combined with \eqref{it:534} from Lemma \ref{lm:F}.
	More precisely, we process the vertices in $X$ one by one, and whenever  dealing 
	with some $x\in X$ we pick an $x$-absorber in $P_A$ that has not been used before 
	and use its absorbing property (as explained after Definition~\ref{dfn:abs}). 
\end{proof}

\section{Almost spanning cycle} \label{sec:cyc}

This section is dedicated to the proof of Proposition~\ref{p:alg}, which is structured 
as follows. In Subsection~\ref{sec:61} we derive an ``approximate version'' 
of Pikhurko's $K_4^{(3)}$-factor theorem (see 
Lemma~\ref{lm:61}) by imitating his proof from \cite{Pi08}. This lemma leads to 
Proposition~\ref{p:alg} in the light of the hypergraph regularity method, which we recall 
in Subsection \ref{sec:62}.

\subsection{\texorpdfstring{$K_4^{(3)}$-tilings}{Tilings}}\label{sec:61}

The subsequent lemma will later be applied to a hypergraph obtained by means of the 
regularity lemma.

\begin{lemma}\label{lm:61}
	Let $t\geq 36$, $0<\alpha <1/4$ and $\alpha \gg \tau$.
	Given a hypergraph $G$ on $t$ vertices such that all but at most  $\tau t^2$ 
	unordered pairs $xy\in V^{(2)}$ of distinct vertices satisfy 
	$d(x,y)\geq (3/4+\alpha)t$, it is possible to delete at 
	most $2\sqrt{\tau}t+14$  vertices and find a $K_4^{(3)}$-factor afterwards.
\end{lemma}

The following proof is similar to Pikhurko's argument establishing \cite{Pi08}*{Theorem~1}.

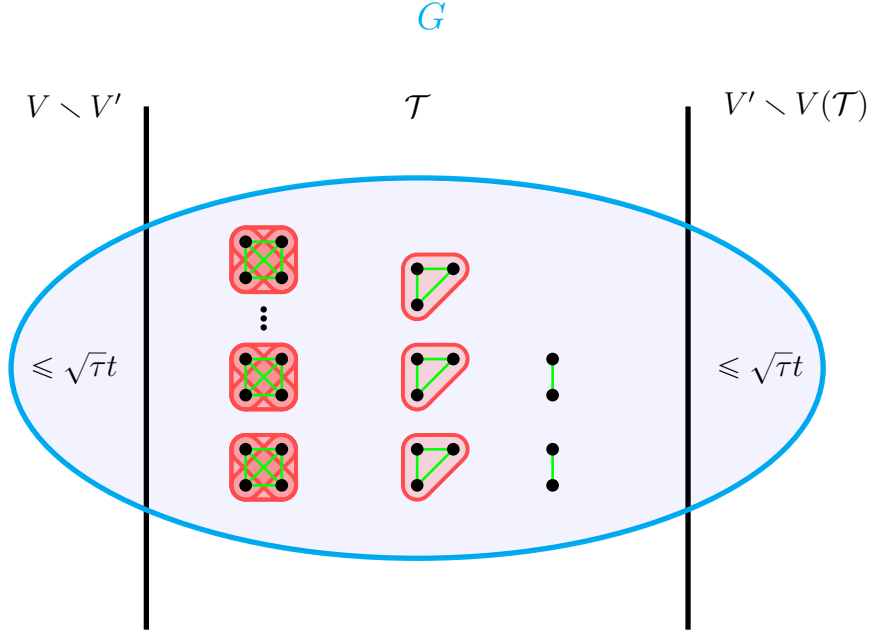
\begin{figure}[ht]
\begin{tikzpicture}[scale=1.2]
\coordinate (do1) at (2.3,2.65);	
\coordinate (do2) at (2.3,2.75);
\coordinate (do3) at (2.3,2.55);
	
	\coordinate (a') at (2.5,1.8);
		\coordinate (c') at (2.1,2.2);
		\coordinate (d') at (2.1,1.8);
	\coordinate (b') at (2.5,2.2);
	
	\coordinate (a) at (2.5,0.8);
		\coordinate (c) at (2.1,1.2);
		\coordinate (d) at (2.1,0.8);
	\coordinate (b) at (2.5,1.2);
	
	\coordinate (a2) at (2.5,3.1);
		\coordinate (c2) at (2.1,3.5);
		\coordinate (d2) at (2.1,3.1);
	\coordinate (b2) at (2.5,3.5);
	
	\coordinate (a3) at (4,2.8);
		\coordinate (c3) at (4.4,3.2);
	\coordinate (b3) at (4,3.2);
	
	\coordinate (a4) at (4,1.8);
		\coordinate (c4) at (4.4,2.2);
	\coordinate (b4) at (4,2.2);
	
	\coordinate (a5) at (4,0.8);
		\coordinate (c5) at (4.4,1.2);
	\coordinate (b5) at (4,1.2);
	
	\coordinate (a6) at (5.5,0.8);
	\coordinate (b6) at (5.5,1.2);
	
	\coordinate (a7) at (5.5,1.8);
	\coordinate (b7) at (5.5,2.2);
	
	\coordinate (b8) at (7,5);
		\coordinate (a8) at (7,-0.8);
		\coordinate (b9) at (1,5);
		\coordinate (a9) at (1,-0.8);
		
	\coordinate (g) at (4.5,6);
	\coordinate (t) at (4,5);
	\coordinate (t3) at (0.2,5);
	\coordinate (t4) at (8.2,5);
	\coordinate (t1) at (7.8,2.1);
	\coordinate (t2) at (0.2,2.1);

	\begin{pgfonlayer}{front}

		\fill (a) circle (2pt);
		\fill (do1) circle (1pt);
		\fill (do2) circle (1pt);
		\fill (do3) circle (1pt);
		\fill (b) circle (2pt);
		\fill (c) circle (2pt);
		\fill (d) circle (2pt);
		\fill (a') circle (2pt);
		\fill (b') circle (2pt);
		\fill (c') circle (2pt);
		\fill (d') circle (2pt);
		\fill (a2) circle (2pt);
		\fill (b2) circle (2pt);
		\fill (c2) circle (2pt);
		\fill (d2) circle (2pt);
		\fill (a3) circle (2pt);
		\fill (b3) circle (2pt);
		\fill (c3) circle (2pt);
		\fill (a4) circle (2pt);
		\fill (b4) circle (2pt);
		\fill (c4) circle (2pt);
		\fill (a5) circle (2pt);
		\fill (b5) circle (2pt);
		\fill (c5) circle (2pt);
		\fill (b6) circle (2pt);
		\fill (a6) circle (2pt);
		\fill (b7) circle (2pt);
		\fill (a7) circle (2pt);

		\node at ($(g)+(180:0.34)$) { \large  \textcolor{cyan}{$G$}};
\node at ($(t)+(0:0)$) { $\mathcal{T}$};
\node at ($(t3)+(0:0)$) { $V\setminus V'$};
\node at ($(t4)+(0:0)$) { $V'\setminus V(\mathcal{T})$};
		\node at ($(t1)+(0:0)$) {$\leq \sqrt{\tau} t$};
		\node at ($(t2)+(0:0)$) {$\leq \sqrt{\tau} t$};
												
	\end{pgfonlayer}
	
	\begin{pgfonlayer}{background}

				\draw[black, line width=2pt] (a9) -- (b9);
		\draw[black, line width=2pt] (a8) -- (b8);
		\draw[cyan, line width=2pt] (4,2.1) ellipse (128pt and 60pt);
\fill[blue, opacity=0.05] (4,2.1) ellipse (128pt and 60pt);
	\end{pgfonlayer} 

\qedge{(a3)}{(b3)}{(c3)}{4.5pt}{1.5pt}{red!70!white}{red!70!white,opacity=0.2};
\qedge{(a4)}{(b4)}{(c4)}{4.5pt}{1.5pt}{red!70!white}{red!70!white,opacity=0.2};
\qedge{(a5)}{(b5)}{(c5)}{4.5pt}{1.5pt}{red!70!white}{red!70!white,opacity=0.2};
\qedge{(b2)}{(a2)}{(c2)}{4.5pt}{1.5pt}{red!70!white}{red!70!white,opacity=0.2};
\qedge{(b2)}{(d2)}{(c2)}{4.5pt}{1.5pt}{red!70!white}{red!70!white,opacity=0.2};
\qedge{(b2)}{(a2)}{(d2)}{4.5pt}{1.5pt}{red!70!white}{red!70!white,opacity=0.2};
\qedge{(c2)}{(a2)}{(d2)}{4.5pt}{1.5pt}{red!70!white}{red!70!white,opacity=0.2};

\qedge{(b)}{(a)}{(c)}{4.5pt}{1.5pt}{red!70!white}{red!70!white,opacity=0.2};
\qedge{(b)}{(d)}{(c)}{4.5pt}{1.5pt}{red!70!white}{red!70!white,opacity=0.2};
\qedge{(b)}{(a)}{(d)}{4.5pt}{1.5pt}{red!70!white}{red!70!white,opacity=0.2};
\qedge{(c)}{(a)}{(d)}{4.5pt}{1.5pt}{red!70!white}{red!70!white,opacity=0.2};

\qedge{(b')}{(a')}{(c')}{4.5pt}{1.5pt}{red!70!white}{red!70!white,opacity=0.2};
\qedge{(b')}{(d')}{(c')}{4.5pt}{1.5pt}{red!70!white}{red!70!white,opacity=0.2};
\qedge{(b')}{(a')}{(d')}{4.5pt}{1.5pt}{red!70!white}{red!70!white,opacity=0.2};
\qedge{(c')}{(a')}{(d')}{4.5pt}{1.5pt}{red!70!white}{red!70!white,opacity=0.2};

\draw[green, line width=1pt] (a) -- (c);
		\draw[green, line width=1pt] (a) -- (b);
		\draw[green, line width=1pt] (a) -- (d);
		\draw[green, line width=1pt] (c) -- (b);
		\draw[green, line width=1pt] (d) -- (b);
		\draw[green, line width=1pt] (c) -- (d);
		
		\draw[green, line width=1pt] (a') -- (c');
		\draw[green, line width=1pt] (a') -- (b');
		\draw[green, line width=1pt] (a') -- (d');
		\draw[green, line width=1pt] (c') -- (b');
		\draw[green, line width=1pt] (d') -- (b');
		\draw[green, line width=1pt] (c') -- (d');
		
		\draw[green, line width=1pt] (a2) -- (c2);
		\draw[green, line width=1pt] (a2) -- (b2);
		\draw[green, line width=1pt] (a2) -- (d2);
		\draw[green, line width=1pt] (c2) -- (b2);
		\draw[green, line width=1pt] (d2) -- (b2);
		\draw[green, line width=1pt] (c2) -- (d2);
		
		\draw[green, line width=1pt] (a3) -- (b3);
		\draw[green, line width=1pt] (c3) -- (b3);
		\draw[green, line width=1pt] (a3) -- (c3);
	
		\draw[green, line width=1pt] (a4) -- (b4);
		\draw[green, line width=1pt] (c4) -- (b4);
		\draw[green, line width=1pt] (a4) -- (c4);
		
		\draw[green, line width=1pt] (a5) -- (b5);
		\draw[green, line width=1pt] (c5) -- (b5);
		\draw[green, line width=1pt] (a5) -- (c5);
		
		\draw[green, line width=1pt] (a6) -- (b6);
		\draw[green, line width=1pt] (a7) -- (b7);
		
	\end{tikzpicture}
	\caption {Example of a tiling $\cT$ with maximal weight, 
	where \textcolor{green}{good pairs} are indicated by green edges.}
	\label{fig:tiling}
	\end{figure}
	
\begin{proof}	
	Let us call a pair of vertices {\it bad} if its pair-degree is smaller than $(3/4+\alpha)t$. 
	Moreover we will call a subhypergraph of $G$ {\it good} if it does not contain any bad pair 
	of vertices.
	
	For $X\subseteq V$ we denote number of bad pairs in $X^{(2)}$ by $B(X)$. Let 
	$V'\subseteq V$ be maximal with the property $B(V)-B(V')\ge |V\sm V'|\sqrt{\tau} n$. 
	Such a set exists because $V$ has this property. Note that $B(V)\le \tau n^2$ entails 
	$|V\sm V'|\le \sqrt{\tau} n$. Let $G'=(V',E')$ be the subhypergraph of $G$ induced by $V'$.
	Observe that for each $v\in V'$ the maximality of $V'$ entails
	\begin{align*}
		B(V')-B(V'\sm \{x\}) &=\bigl(B(V)-B(V'\sm \{x\})\bigr)-\bigl(B(V)-B(V')\bigr) \\
		& < (|V\sm V'|+1)\sqrt{\tau} n - |V\sm V'|\sqrt{\tau} n = \sqrt{\tau} n\,,
	\end{align*}
	meaning that in $G'$ every vertex is in less than $\sqrt{\tau} n$ bad pairs. 
		
	Let $\mathcal{F}$ be a set of hypergraphs. By an {\it $\mathcal{F}$-tiling} in $G$ we mean 
	a collection of vertex-disjoint good subgraphs, each of which is isomorphic to some member of 
	$\mathcal{F}$. Moreover let $w_2=2$, $w_3=6$, and $w_4=11$ be weight factors.

 	In the following we will consider a $\bigl\{K_2^{(3)},K_3^{(3)},K_4^{(3)}\bigr\}$-tiling 
	$\cT$ in $G'$ that maximises the weight function $w(\cT)=w_2\ell_2+w_3\ell_3+w_4\ell_4$, 
 	where $\ell_i$ denotes the number of copies of $K_i^{(3)}$ in $\cT$. 
 
	At most $\sqrt{\tau} t+1$ vertices of $V'$ are missed by the tiling $\cT$. 
	Indeed, otherwise we find a good subgraph isomorphic to $K_2^{(3)}$ not in the tiling, 
	since every vertex in $V'$ is in at most $\sqrt{\tau}t$ bad pairs. Because $w_2>0$ this 
	is a contradiction to the maximality of $\cT$.
 
	We say a hypergraph $F\in \cT$ makes a {\it connection} with the vertex 
	$x\in V'\setminus V(F)$ (denoted by $(F,x)\in \mathcal{C}$) if $|V(F)|\leq 3$ and 
	$V(F)\cup \{ x \}$ spans a complete good hypergraph. Examining the properties of  
	connections, we get the following results.
	
	\begin{enumerate}
 		\item[(A)] A $K_i^{(3)}$-subgraph $F\in \cT$ with $i\leq 3$ can only make a connection 
		to a vertex $x$ that belongs to a $K_j^{(3)}$-subgraph of $\cT$ with $j>i$.
	\end{enumerate}
 	
	Otherwise moving $x$ to $F$ would increase the weight of $\cT$, since $w_4+w_2-2w_3=1$, 
	$w_4-w_2-w_3=3$, $w_3-2w_2=2$, and all other possible weight changes are positive as well.
 	
	\begin{enumerate}
  		\item[(B)] Each $K_2^{(3)}$-subgraph $F$ in $\cT$ makes at least 
		$(\frac{3}{4}+\frac{\alpha}{2}) t$ connections.
 	\end{enumerate}
 	
	Let $\{a,b\}$ be the vertex set of $K_2^{(3)}$-subgraph $F$ of $\cT$. The subgraph $F$ 
	makes a connection with a vertex $x\in V'\setminus V(F)$ if $abx\in E(G)$ and $ab,ax,bx$ 
	are good pairs. Recalling that $ab$ is a good pair due to the definition of tiling, 
	we can relax the second condition to $ax,bx$ being good pairs.
 	There are at least $(\frac{3}{4}+\alpha -\sqrt{\tau}) t$ vertices in $V'\setminus V(F)$ 
	that form an edge  with $ab$ in $G$. Since every vertex in $V'$ is in at most 
	$\sqrt{\tau }t$ bad pairs, at most $2\sqrt{\tau }t$ vertices, which form an edge with 
	$ab$ in $G$, can fail the second condition. Thus, every $K_2^{(3)}$-subgraph~$F$ 
	of~$\cT$ makes at least $(\frac{3}{4}+\alpha -3\sqrt{\tau })t$ connections, which due 
	to $\tau<\frac{\alpha^2}{36}$ is more than $(\frac{3}{4}+\frac{\alpha}{2}) t$.
 	
	\begin{enumerate}
 		\item[(C)] Every $K_3^{(3)}$-subgraph~$F$ in $\cT$ makes at least $(\frac{1}{4}+\alpha)t$
		 connections.
 	\end{enumerate} 
 	
	For each $K_3^{(3)}$-subgraph~$F$ of $\cT$ there are by (B) at least $(\frac{9}{4}+\alpha)t$ 
	edges that intersect it in exactly two vertices and consists of no bad pairs.
	Let $c$ denote the number of connections made by a $K_3^{(3)}$-subgraph of $\cT$. 
	Thus, we get
		\[
		\Big(\frac{9}{4}+\alpha \Big)t\leq 3c+2(t-3-c)\, ,
	\]
		i.e.,
	\[
		\Big(\frac{9}{4}+\alpha \Big)t-2t+6\leq c\,.
	\]
		
	\begin{enumerate}
		\item[(D)] $\ell_3\leq 3$.
	\end{enumerate}  
	
	Otherwise let $F_1, F_2, F_3, F_4$ be $K_3^{(3)}$-subgraphs in $\cT$. Due to (A) all connections 
	made by a $F_i$ belong to a $K_4^{(3)}$-subgraph of $\cT$. 
	An upper bound for the number of $K_4^{(3)}$ in $\cT$ is~$\lfloor t/4 \rfloor$. 
	Since 
		\[
		4\Big (\frac{1}{4}+\alpha \Big) t>4\lfloor t/4 \rfloor \,,
	\]
		the vertices of some $K_4^{(3)}$-subgraph $F$ of $\cT$ make at least $5$ connections 
	with $F_1,F_2,F_3,F_4$. Therefore we find two distinct vertices $x,y\in V(F)$ and 
	$i,j\in [4]$ with $i\neq j$, such that $(F_i,x), (F_j,y)\in \mathcal{C}$. 
	Moving $x$ to $F_i$ and $y$ to $F_j$ and thereby reducing $F$ to a $K_2^{(3)}$ would 
	increase the weight of $\cT$, since $2(w_4-w_3)+(w_2-w_4)=1$. Thus, we get a contradiction 
	to the maximality of $\cT$.
	
	\smallskip
	
	{\it Case 1.} $\ell_2\geq 3$
 	
	\smallskip
	
	Let $F_1, F_2, F_3$ be $K_2^{(3)}$-subgraphs in $\cT$.
    
    \begin{enumerate}
 		\item[(E)] There is no $K_3^{(3)}$-subgraph $F\in \cT$ with the property that 
		$F_1,F_2,F_3$ make more than~$3$ connections to $F$.
 	\end{enumerate}
 	
	Otherwise we could find distinct vertices $x,y\in V(F)$ and $i,j\in [3]$ with $i\neq j$, 
	such that $(F_i,x),(F_j,y)\in \mathcal{C}$. Moving $x$ to $F_i$ and $y$ to $F_j$ and 
	thereby eliminating $F$ would increase the weight of $\cT$, since $2(w_3-w_2)-w_3=2$. 
	Thus, we get a contradiction to the maximality of $\cT$.
	
	\begin{enumerate}
		\item[(F)] There is no $K_4^{(3)}$-subgraph $F\in \cT$ with the property that 
		$F_1,F_2,F_3$ make more than~$8$ connections to $F$.
  	\end{enumerate}
 	
	Otherwise we could find distinct vertices $x_1,x_2,x_3\in V(F)$, such that 
	$(F_i,x_i)\in \mathcal{C}$ for every~$i\in [3]$. This is because every bipartite 
	graph with nine edges and partition classes of size~$3$ and~$4$ contains a matching 
	of size $3$. Moving each $x_i$ to $F_i$ and thereby eliminating~$F$ would increase 
	the weight of $\cT$, since $3(w_3-w_2)-w_4=1$. Thus, we get a contradiction to the 
	maximality of $\cT$.

	Finally (A), (E), and (F) imply an upper bound of $3\ell_3+8\ell_4$ on the number of 
	connections created by $F_1,F_2,F_3$. Because of (C) this leads to 
		\[
		3\Big (\frac{3}{4}+\frac{\alpha}{2}\Big )t\leq 3\ell_3+8\ell_4\,.
	\]
		Since $\ell_3\leq 3$ and $\ell_4\leq \lfloor t/4\rfloor$, we have
		\[
		 \Big(\frac{9}{4}+\frac{3}{2}\alpha \Big)t\leq 9+ 8 \lfloor t/4\rfloor\,, 
	\]
		which contradicts $t\geq 36$.

	\smallskip
	
	{\it Case 2.} $\ell_2\leq 2$
	
	\smallskip
	
	We have deleted $\sqrt{\tau}t$ vertices from $G$ to obtain the graph $G'$, 
	another $\sqrt{\tau}t+1$ vertices can be missed by the tiling $\cT$, and at 
	most  $2\ell_2+3\ell_3\leq 13$ vertices of $V(\cT)$ are not covered by $K_4^{(3)}$ 
	subgraphs. Therefore it is possible to delete at most $2\sqrt{\tau}t+14$ vertices 
	and find a $K_4^{(3)}$-factor afterwards.
\end{proof}

\subsection{Hypergraph regularity method}\label{sec:62}

We denote by $K(X,Y)$ the complete bipartite graph with vertex partition $X\dcup Y$. 
For a bipartite graph $P=(X \dcup Y ,E)$ we say it is $(\delta_2,d_2)$-{\it quasirandom} if
\[ 
	\big |e(X',Y')-d_2|X'||Y'|\big |\leq \delta_2|X||Y|
\]
holds for all subsets $X'\subseteq X$ and $Y'\subseteq Y$, where $e(X',Y')$ denotes the 
number of edges in~$P$ with one vertex in $X'$ and one in $Y'$. Given a $k$-partite graph
$P=(X_1\dcup \ldots  \dcup X_k,E)$ with $k\geq 2$ we say $P$ 
is {\it $(\delta_2,d_2)$-quasirandom}, 
if all naturally induced bipartite subgraphs $P[X_i,X_j]$ are $(\delta_2, d_2)$-quasirandom.
Moreover, for a tripartite graph $P=(X\dcup Y\dcup Z,E)$ we denote by
\[
	\mathcal{K}_3(P)=\big \{ \{x,y,z\}\subseteq X\cup Y\cup Z\colon xy,xz,yz\in E\big \}
\]
the triples of vertices in $P$ spanning a triangle. For a $(\delta_2,d_2)$-quasirandom 
tripartite graph $P=(X\dcup Y \dcup Z,E)$  the so-called {\it triangle counting lemma} 
(see e.g. the survey article~\cite{RS12}*{Theorem~18} or the research 
monograph~\cite{Lov12}*{Lemma 10.24}) 
implies that
\begin{equation}
	d_2^3|X||Y||Z|-3\delta_2|X||Y||Z| 
	\leq 
	|\mathcal{K}_3(P)|\leq d_2^3|X||Y||Z|+3\delta_2|X||Y||Z|\, .
\end{equation}

\begin{dfn}
	Given a $3$-uniform hypergraph $H=(V,E_H)$ and a tripartite graph $P=(X\dcup Y\dcup Z,E)$ 
	with $X\cup Y\cup Z\subseteq V$ we say $H$ is $(\delta_3,d_3)$-{\it quasirandom with 
	respect to} $P$ if for every tripartite subgraph $Q\subseteq P$ we have
		\[
		\big ||E_H\cap \mathcal{K}_3(Q)|-d_3|\mathcal{K}_3(Q)|\big |
		\leq 
		\delta_3 |\mathcal{K}_3(P)| \,.
	\]
		
	Furthermore, we say $H$ is $\delta_3$-{\it quasirandom with respect to} $P$, if it is 
	$(\delta_3,d_3)$-quasirandom for some $d_3\geq 0$.
\end{dfn}

We define the {\it relative density} of $H$ with respect to $P$ by
\[
	d(H|P)=\dfrac{|E_H\cap \cK_3(P)|}{|\cK_3(P)|} \,,
\]
where $d(H|P)=0$ if $\cK_3(P)=\emptyset$.

A refined version of the regularity lemma (see~\cite{RS07}*{Theorem~2.3}) states the following.

\begin{lemma}[Regularity Lemma]\label{thm:regularity}
	For every $\delta_3>0$, every $\delta_2\colon \NN \rightarrow (0,1]$, 
	and every $t_0\in \NN$ there exists an integer $T_0$ such that for 
	every $n\geq t_0$ and every $n$-vertex 3-uniform hypergraph $H=(V,E_H)$ the 
	following holds.
	
	There are integers $t$ and $\ell$ with $t_0\leq t\leq T_0$ and $\ell\leq T_0$ and 
	there exists a vertex partition $V_0\dcup V_1 \dcup \ldots \dcup V_t=V$ and for all 
	$1\leq i<j\leq t$ there exists a partition
		\[
		\mathcal{P}^{ij}=\{P_\alpha ^{ij}=(V_i \dcup V_j,E_\alpha^{ij})\colon1\leq \alpha \leq \ell \}
	\]
		of the edge set of the complete bipartite graph $K(V_i,V_j)$ satisfying the 
	following properties 
	\begin{enumerate}
		\item \label{it:reg1} $|V_0|\leq \delta_3 n$ and $|V_1|=\ldots =|V_t|$,
		\item \label{it:reg2}for every $1\leq i<j\leq t$ and $\alpha \in [\ell]$ 
			the bipartite graph $P_\alpha^{ij}$ is $(\delta_2(\ell),1/\ell)$-quasirandom, and
		\item \label{it:reg3} $H$ is $\delta_3$-quasirandom w.r.t $P_{\alpha \beta\gamma}^{ijk}$ 
			for all but at most $\delta_3t^3\ell^3$ tripartite graphs
						\[
				P_{\alpha \beta \gamma}^{ijk}
				=
				P_{\alpha}^{ij}\dcup P_{\beta}^{ik}\dcup P_{\gamma}^{jk}
				=
				(V_i \dcup V_j\dcup V_k, E_\alpha ^{ij} \dcup E_\beta ^{ik}\dcup E_\gamma^{jk})
				\,,
			\]
					with $1\leq i<j<k\leq t$ and $\alpha, \beta, \gamma \in [\ell]$.
	\end{enumerate}
\end{lemma}

The tripartite graphs  $P_{\alpha \beta \gamma}^{ijk}$ appearing in \eqref{it:reg3} are 
usually called {\it triads}. Next we state the following consequence of the embedding 
lemma~\cite{NPRS09}*{Corollary~2.3} (see also~\cite{CR17}*{Theorem~5.3}).

\begin{lemma}\label{lm:alang}
	Given $Q\in \NN$ and $d_3>0$, there exist $\delta_3>0$, and functions 
	$\delta_2\colon \NN \rightarrow (0,1]$ and $N\colon\NN\rightarrow \NN$, 
	such that that the following holds for every $\ell\in \NN$. 
	
	Let $P=(V_1\dcup V_2 \dcup V_3 \dcup V_4, E_P) $ be a  $4$-partite graph with 
	$|V_1|=\ldots =|V_4|=n\geq N(\ell)$ such that $P ^{ij}=(V_i \dcup V_j,E^{ij})$ is 
	$(\delta_2(\ell), 1/\ell)$-quasirandom for every pair $ij\in [4]^{(2)}$. 
	Suppose~$H$ is a $4$-partite, $3$-uniform hypergraph with vertex classes $V_1,\ldots ,V_4$, 
	which satisfies for every $ijk\in [4]^{(3)}$ that $H$  is 
	$(\delta_3, d_{ijk})$-quasirandom w.r.t. 
	the tripartite graphs $P^{ijk}=P^{ij}\dcup P^{ik}\dcup P^{jk}$ for some 
	$d_{ijk}\geq d_3$.
	Then there exists a squared path with $Q$ vertices in $H$.
\end{lemma}

An iterative application of this lemma leads to the following statement. 

\begin{lemma}\label{lm:67}
	Given $Q\in \NN$ with $Q \equiv 0 \pmod 4$, $d_3>0$, and $\nu>0$. 
	There exist $\delta_3>0$, $ \delta_2\colon\NN\rightarrow (0,1)$, 
	and $N\colon\NN \rightarrow \NN$, such that the following 
	holds for every $\ell\in \NN$. Let $P=(V_1\dcup V_2 \dcup V_3 \dcup V_4, E_P)$ 
	be a  $4$-partite graph with $|V_1|=\ldots =|V_4|=n\geq N(\ell)$ and 
	let $P^{ij}=(V_i \dcup V_j,E^{ij})$ be $(\delta_2(\ell), 1/\ell)$-quasirandom
	for every $ij\in[4]^{(2)}$. 
	Suppose that $H$ is a $3$-uniform hypergraph, which satisfies for every $ijk\in [4]^{(3)}$ 
	that $H$ is $(\delta_3, d_{ijk})$-quasirandom with respect to the tripartite graph 
	$P^{ijk}=P^{ij}\dcup P^{ik}\dcup P^{jk}$ for some $d_{ijk}\geq d_3$.
   	Then all but at most $\nu n$ vertices of $V_1\dcup \ldots \dcup V_4$ 
	can be covered by vertex-disjoint squared paths with $Q$ vertices each. 
\end{lemma}

\begin{proof}
	Let $\delta_3^*>0$, $\delta_2^*\colon\NN\rightarrow (0,1]$, 
	$N^*\colon\NN\rightarrow \NN$ be the number and functions 
	obtained by applying Lemma~\ref{lm:alang} to $Q$ and $d_3/2$. 
	Define 
		\[
		\delta_3=\frac{\delta_3^* \nu ^3 }{128},\,\,\,
				\delta_2(\ell)
		=
		\min \Big (\frac{\delta_2^*(\ell )\nu^2}{16}, \frac{\nu^2}{144 \ell^3}\Big ),\,\,\,
		N(\ell)= \Big \lceil \frac{4 N^*(\ell)}{ \nu} \Big \rceil 
	\]
		for each $\ell \in \NN$. Let $P=(V_1\dcup V_2 \dcup V_3 \dcup V_4, E_P) $ 
	and $H$ be as described above for some~$\ell \in \NN$. Consider a maximal 
	collection $S_1, \ldots , S_m$ of vertex-disjoint squared paths on $Q$ vertices in~$H$. 
	For $i\in [4]$ let $V'_i\subseteq V_i$ denote the set of vertices not belonging to any 
	of these paths. Due to $4\mid Q$ the sets $V'_1, \ldots ,V'_4$ have the same size, 
	say $n^*$. If $n^*< \nu n/4$ we are done, so assume from now on that $n^*\geq \nu n/4$. 
	Then our choice of $\delta_2(\ell )$ implies that the bipartite graphs 
	$P ^{ij}[V'_i \dcup V'_j]$ are $(\delta_2^{**}(\ell ), 1/\ell )$-quasirandom, 
	where $\delta_2^{**}(\ell)=\min ( \delta_2^{*}(\ell), \frac{1}{9\ell^3} )$. 
	So by Lemma~\ref{lm:alang} we get a contradiction to the maximality of $m$ provided 
	we can show that~$H$ is $(\delta_3^*, d_{ijk})$-quasirandom w.r.t. the subtriads $P_*^{ijk}$ 
	of $P^{ijk}$ induced by $V'_i\cup V'_j\cup V'_k$. 
	This is indeed the case, since the triangle counting lemma yields that
		\begin{align*}
		|\mathcal{K}_3(P^{123})|&\leq |V_1||V_2||V_3| \bigl(1/\ell^3+3 \delta_2(\ell)\bigr) \\
		&\overset{n^*\ge \nu n/4}{\leq} 
			\dfrac{4^3|V'_1||V'_2||V'_3|}{ \nu^3} (1/\ell^3+3 \delta_2(\ell))\\
		& \leq \dfrac{64\cdot \mathcal{K}_3(P[V'_1, V'_2, V'_3])}{\nu^3} \cdot 
			\dfrac{(1/\ell^3+3 \delta_2(\ell))}{(1/\ell^3-3 \delta^{**}_2(\ell))}\\
		&\leq 128 \cdot \dfrac{\mathcal{K}_3(P^{123}_*)}{\nu^3} \,,
	\end{align*}
			i.e., 
		\[
			\delta_3|\mathcal{K}_3(P^{123})|\le \delta_3^*|\mathcal{K}_3(P^{123}_*)|\,,
		\]
	and the same argument applies to every other triple ${ijk\in [4]^{(3)}}$.
\end{proof}

\subsection{Vertex-disjoint squared paths with \texorpdfstring{$Q$}{\it Q} vertices}
\label{sec:64}

Next we restate and prove Propostion~\ref{p:alg}. 

\begin{prop}\label{thm:67}
	Given $\alpha, \mu >0 $ and $Q\in \NN$ there exists $n_0\in \NN$ such that 
	in every hypergraph~$H$ with $v(H)=n\geq n_0$ and $\delta_2(H)\geq (3/4+\alpha)n$ all but 
	at most $\mu n$ vertices of~$H$ can be covered by vertex-disjoint squared paths with $Q$ 	
	vertices.
\end{prop}

\begin{proof}
	As we could replace $Q$ by $4Q$ if necessary we may suppose that $Q$ is a multiple of~$4$.
	Pick sufficiently small $d_3, \nu ,\tau \ll \alpha, \mu$ and let $\delta_3>0$, 
	$\delta_2\colon\NN \rightarrow (0,1)$, 
	$N\colon\NN \rightarrow \NN$ be the number and functions obtained 
	by applying Lemma \ref{lm:67} to $Q$, $\nu$, and $d_3$.
	W.l.o.g. $\delta_3$, $\delta_2(\cdot)$ are sufficiently small, such that 
	$\delta_3 \ll \alpha, \tau$, and $\delta_2(\ell) \ll \alpha, \ell^{-1}, \tau$. 
	For $t_0$ sufficiently large we can use Lemma \ref{thm:regularity} 
	with $\delta_3, \delta_2, t_0$ and  get an integer $T_0$. Finally we let 
	$n_0$ be sufficiently large. 

	Now let $H$ be a $3$-uniform hypergraph with $v(H)=n\geq n_0$ and 
	$\delta_2(H)\geq (\frac{3}{4}+\alpha)n$. Due to Lemma \ref{thm:regularity} 
	there exists a vertex partition $V_0\dcup V_1\dcup \ldots \dcup V_t=V$ and 
	pair partitions 
		\[
		\mathcal{P}^{ij}
		=
		\{P_\alpha ^{ij}=(V_i \dcup V_j,E_\alpha^{ij})\colon 1\leq \alpha \leq \ell \}
	\]
		of the complete bipartite graphs $K(V_i,V_j)$ for $1\leq i<j\leq t$ 
	satisfying \eqref{it:reg1}--\eqref{it:reg3}.

	We call a triad $P^{ijk}_{\alpha \beta \gamma}$ {\it dense} if 
	$d(H|P^{ijk}_{\alpha \beta \gamma})\geq \alpha /10$. 
	For every pair $i_*j_*\in [t]^{(2)}$ and 
	every $\lambda \in [\ell]$ 
	we denote the set of dense triads involving $V_{i_*}$, $V_{j_*}$, 
	and $P^{i_*j_*}_\lambda $ by $\mathcal{D}_\lambda (i_*,j_*)$.
	
	\begin{claim}\label{cl:67}
		For every $i_*j_*\in [t]^{(2)}$ we have 
		$|\mathcal{D}_\lambda (i_*,j_*)|\geq (\frac{3}{4}+\frac{\alpha}{2})\ell^2 t$.
	\end{claim}
	
	\begin{proof}
		Notice that Lemma \ref{thm:regularity}\eqref{it:reg1} yields
				\begin{equation}\label{eq:62a}
			\dfrac{n(1-\delta_3)}{t}\leq |V_k|\leq \dfrac{n}{t}
		\end{equation}
				for every $k \in [t]$. Appealing to the $(\delta_2(\ell ), 1/\ell)$-quasirandomness 
		of $P^{i_*j_*}_\lambda $ we infer
				\begin{align*}
			|E^{i_*j_*}_\lambda | &
				\geq \Big (\dfrac{1}{\ell}-\delta_2(\ell)\Big ) |V_{i_*}| |V_{j_*}|  \\
			&\geq \Big (\dfrac{1}{\ell}-\delta_2(\ell)\Big ) 
				\Big( \dfrac{(1-\delta_3)n}{t}\Big)^2 \, .
		\end{align*}
				Together with the lower bound on $\delta_2 (H)$ and $|V_0|\leq \delta_3 n$ 
		it follows that
				\begin{equation}\label{eq:63}
			\Big (\dfrac{1}{\ell}-\delta_2(\ell)\Big )
			\Big( \dfrac{(1-\delta_3)n}{t}\Big)^2 \Big( \dfrac{3}{4}+\alpha -\delta_3 \Big )n 
			\leq 
			\sum_{xy\in E_\lambda ^{i_*j_*}} |N(x,y)\setminus V_0|\, .
		\end{equation}
				
		On the other hand we can derive an upper bound on the right side by counting the 
		edges in each triad using $E^{i_*j_*}_\lambda $ separately.
		Due to the triangle counting lemma and \eqref{eq:62a} each such triad contains 
		at most
				\[
		 	\Big (\dfrac{1}{\ell^3}+3\delta_2 (\ell) \Big ) \Big (\dfrac{n}{t}\Big )^3
		\]
				triangles. Therefore we have
				\[
			\sum_{xy\in E_\lambda ^{i_*j_*}} |N(x,y)\setminus V_0|
			\leq 
			t\ell^2 \dfrac{\alpha}{10} \Big(\dfrac{n}{t}\Big )^3
			\Big( \dfrac{1}{\ell^3}+3\delta_2(\ell)\Big)  
			+|\mathcal{D}_\lambda (i_*,j_*)| \Big(\dfrac{n}{t}\Big)^3
			\Big(\dfrac{1}{\ell^3}+3\delta_2(\ell)\Big ) \,.
		\]
				Combined with \eqref{eq:63} this leads because of $\delta_3\ll \alpha$ and 
		$\delta_2\ll \alpha /\ell^3$ to
				\[
			|\mathcal{D}_\lambda (i_*,j_*)| \geq (3/4+\alpha /2)\ell^2 t\,. \qedhere
		\]
			\end{proof}

	For every $f\colon[t]^2\rightarrow [\ell]$ we define a hypergraph $J_f$ on the vertex set 
	$[t]$ such that a $3$-element set $\{i, j ,k\}$ is an edge of $J_f$ if the triad 
	$P_{f(ij) f(ik) f(jk)}^{ijk}$ is dense and $H$ is $\delta_3$-quasirandom w.r.t. this triad.

	\begin{claim}\label{cl:68}
		There is $f\colon[t]^{(2)}\rightarrow [\ell]$ such that all but at most $\tau t^2$ 
		pairs $ij\in [t]^{(2)}$ have at least pair-degree $(\frac{3}{4}+\frac{\alpha}{8})t$ 
		in $J_f$.
	\end{claim}

	\begin{proof}
 		Let $D_f$ be the hypergraph on $[t]$ whose edges are the triples $ijk$ such that 
		the triad $P_{f(ij) f(ik) f(jk)}^{ijk}$ is dense, and let $R_f$ be the 
		hypergraph consisting 
		of all sets $\{i,j ,k\}$  such that  $H$ is $\delta_3$-quasirandom with respect to the triad 
		$P_{f(ij) f(ik) f(jk)}^{ijk}$. Clearly, $J_f=D_f\cap R_f$. We will show that 
		if we choose $f$ uniformly at random, then with positive probability 
		$E(\overline{R_f})\leq 2\delta_3t^3$ and $\delta_2(D_f)\geq (3/4+\alpha /4)t$ hold.
 
 		The expected value of the number of missing edges in $R_f$ is
				\[
			\EE( E(\overline{R_f}))\leq \dfrac{1}{\ell^3}\cdot \delta_3 t^3 \ell^3
			=\delta_3 t^3\,,
		\]
				since by Lemma \ref{thm:regularity}\eqref{it:reg3} there are at most 
		$\delta_3 t^3 \ell^3 $ irregular triads.
		Thus, due to Markov's inequality 
				\begin{equation}\label{eq:62}
			\PP(E(\overline{R_f})>2\delta_3t^3)<\dfrac{\delta_3t^3}{2\delta_3t^3}
			=\dfrac{1}{2} \,.
		\end{equation}
		
		Now fix a pair $i_*j_*\in [t]^{(2)}$. Estimating the expected value 
		of $d_{D_f}(i_*,j_*)$, we get for every 
		$\lambda\in [\ell]$ that 
				\begin{align*}
			&\EE\big (d_{D_f}(i_*,j_*)| f(i_*,j_*)=\lambda \big )\\
			&=\dfrac{1}{\ell^{\binom{t}{2}-1}}\sum_{f\colon[t]^2\rightarrow [\ell], f(i_*,j_*)
				=\lambda} d_{D_f}(i_*,j_*)\\
			&= \dfrac{ |\mathcal{D}_\lambda (i_*,j_*)| }{\ell^2} \,.
		\end{align*}
				By Claim \ref{cl:67} it follows that
				\begin{align*}
			\EE\big (d_{D_f}(i_*,j_*)| f(i_*,j_*)=\lambda \big )\geq (3/4+\alpha /2)t \,.
		\end{align*}
				Moreover, for $f\colon [t]^2 \rightarrow [\ell]$ with $f(i_*,j_*)=\lambda$ 
				the value 
		of $d_{D_f}(i_*,j_*)$ is completely determined by the $2(t-2)$ numbers $f(i,j)$ with 
		$|\{i,j\}\cap \{i_*,j_*\}|=1$ and if one changes one of these $2(t-2)$ values of $f$, 
		then $d_{D_f}(i_*,j_*)$ can change by at most 1.
		Thus, the Azuma-Hoeffding inequality (see, e.g.,~\cite{JLR00}*{Corollary~2.27}) 
		leads to
				\begin{align*}
			\PP\big (d_{D_f}(i_*,j_*)<(3/4+\alpha /4)t\big | f(i_*,j_*)
				=\lambda)
			< \exp\Big (-\dfrac{2(\alpha t/4)^2}{2(t-2)}\Big )\,.
		\end{align*}
				Therefore,
				\begin{equation*}
			\PP\big (d_{D_f}(i_*,j_*)<(3/4+\alpha /4)t\big | f(i_*,j_*)
			=
			\lambda \big )<e^{-\Omega(t)}
		\end{equation*}
				for each $\lambda\in [\ell]$ and hence 
				\begin{equation}
			\PP(d_{D_f}(i_*,j_*)<(3/4+\alpha /4)t)<e^{-\Omega(t)}\,.
		\end{equation}
				Therefore the probability that some pair has a pair-degree less than 
		$(3/4+\alpha /4)t$ is less than $t^2/e^{\Omega (t)}$, which proves that 
		with probability greater then $1/2$ the minimum pair-degree of $D_f$ is at least 
		$(3/4+\alpha /4)t$.
		Together with \eqref{eq:62} this shows that the probability that a function~$f$ 
		fulfills $E(\overline{R_f})\leq 2\delta_3t^3$ and $\delta_2(D_f)\geq (3/4+\alpha /4)t$ 
		is greater than zero.

		From now on let $f\colon [t]^2\rightarrow[\ell ]$ be a fixed function having these 
		two properties. Notice that $D_f\cap R_f$ arise from $D_f$ by deleting at most 
		$2\delta_3t^3$ edges. We can estimate the number $\overline{\tau} t^2$ of pairs, 
		which have afterwards a pair-degree smaller than $ (3/4+\alpha /8)t$, 
		by
				\[
			\overline{\tau} t^2 \alpha t/8\leq 6\delta_3t^3 \,.
		\]
				Thus $\overline{\tau}\leq \frac{48 \delta_3}{\alpha}$ and by our choice of 
		$\delta_3 \ll \alpha,\tau$ it follows that $\overline{\tau}\leq \tau$. 
		In other words, there are indeed at most $\tau t^2$ pairs $ij\in [t]^{(2)}$ 
		whose pair-degree in $J_f$ is smaller than $(\frac{3}{4}+\frac{\alpha}{8})t$.
	\end{proof}

	From now on we will denote the bipartite graph $P^{ij}_{f(i,j)}$ simply by $P^{ij}$, 
	where $f$ is the function obtained in Claim~\ref{cl:68}. Due to Claim \ref{cl:68} 
	we can apply Lemma~\ref{lm:61} to $J_f$ with $\alpha'=\alpha /8$ instead of $\alpha$ 
	and find a $K_4^{(3)}$-factor missing at most $2\sqrt{\tau}t+14$ vertices with 
	$\tau \ll \alpha'$. Since $Q\equiv 0 \pmod 4$, we can apply Lemma~\ref{lm:67} 
	to the ``tetrads'' corresponding to these $K_4^{(3)}$ in the reduced hypergraph. 
	Therefore all but at most 
		\[
		 \frac{n}{t}(2\sqrt{\tau}t+14)+\frac{t}{4} \cdot \nu \cdot \frac{n}{t}+\delta_3 n
		 \leq 
		 \mu n
	\]
		vertices can be covered by vertex-disjoint squared paths with $Q$ vertices.
\end{proof}


\section*{Acknowledgement} 
We would like to thank both referees for a tremendously careful reading of this 
article, which led to some substantial improvements in the presentation.

\noindent
\end{document}